\documentclass[10pt]{article}

\usepackage{amsmath,amssymb,amsthm,comment}
\usepackage{mathtools}
\usepackage{color}
\usepackage[nohead,margin=1.1in]{geometry}
\usepackage{graphicx,psfrag,epstopdf}
\usepackage{verbatim}
\usepackage{algorithmic,algorithm}
\usepackage{url}
\graphicspath{{./pics/}}

\allowdisplaybreaks

\theoremstyle{plain}
\newtheorem{theorem}{Theorem}[section]
\newtheorem{remark}{Remark}[section]

\newtheorem{lemma}{Lemma}[section]

\newtheorem{prop}{Proposition}[section]

\newtheorem{corollary}{Corollary}[section]
\numberwithin{equation}{section}

\title{\bf Regularized Linear Inversion with Randomized Singular Value Decomposition}

\author{Kazufumi Ito\thanks{Department of Mathematics, North Carolina State University, Raleigh, NC 27607, USA (\texttt{kito@ncsu.edu})}\and Bangti Jin\thanks{Department of Computer Science, University College London, Gower Street, London WC1E 6BT, UK. (\texttt{b.jin@ucl.ac.uk,bangti.jin@gmail.com})}}
\date{\today}

\begin{document}

\maketitle

\setlength\abovedisplayskip{5pt}
\setlength\belowdisplayskip{5pt}

\begin{abstract}
In this work, we develop efficient solvers for linear inverse problems
based on randomized singular value decomposition (RSVD). This is achieved by combining
RSVD with classical regularization methods, e.g., truncated singular value
decomposition, Tikhonov regularization, and general Tikhonov regularization with a smoothness penalty.
One distinct feature of the proposed approach is that it explicitly preserves the structure of the regularized solution in
the sense that it always lies in the range of a certain adjoint operator. We provide
error estimates between the approximation and the exact solution under canonical source condition,
and interpret the approach in the lens of convex duality. Extensive numerical experiments are provided to illustrate
the efficiency and accuracy of the approach.
\end{abstract}


\pagestyle{myheadings}
\thispagestyle{plain}

\section{Introduction}

This work is devoted to randomized singular value decomposition (RSVD) for the efficient numerical solution
of the following linear inverse problem
\begin{equation}\label{eqn:lininv}
  A x = b,
\end{equation}
where $A\in\mathbb{R}^{n\times m}$, $x\in\mathbb{R}^m$ and $b\in\mathbb{R}^n$ denote the data formation
mechanism, unknown parameter and measured data, respectively. The data $b$ is generated by $b = b^\dag + e,$
where $b^\dag = Ax^\dag$ is the exact data, and $x^\dag$ and $e$ are the
exact solution and noise, respectively. We denote by $\delta=\|e\|$ the noise level.

Due to the ill-posed nature, regularization techniques are often applied to obtain a stable
numerical approximation. A large number of regularization methods have been developed. The
classical ones include Tikhonov regularization and its variant, truncated singular
value decomposition, and iterative regularization techniques, and
they are suitable for recovering smooth solutions. More recently, general
variational type regularization methods have been proposed to preserve distinct features, e.g.,
discontinuity, edge and sparsity. This work focuses on recovering a smooth solution by Tikhonov regularization and
truncated singular value decomposition, which are still routinely applied
in practice. However, with the advent of the ever increasing data volume, their routine
application remains challenging, especially in the context of massive data
and multi-query, e.g., Bayesian inversion or tuning multiple hyperparameters. Hence, it is still of
great interest to develop fast inversion algorithms.

In this work, we develop efficient linear inversion techniques based on RSVD. Over the last decade,
a number of RSVD inversion algorithms have been developed and analyzed \cite{Gu:2015,HalkoMartinssonTropp:2011,MuscoMusco:2015,
WittenCandes:2015,SzlamTulloch:2017}. RSVD exploits the intrinsic low-rank
structure of $A$ for inverse problems to construct an accurate approximation
efficiently. Our main contribution lies in providing a unified framework for developing fast regularized inversion
techniques based on RSVD, for the following three popular regularization methods: truncated SVD,
standard Tikhonov regularization, and Tikhonov regularization with a smooth penalty. The
main novelty is that it explicitly preserves a certain range condition
of the regularized solution, which is analogous to source condition in regularization
theory \cite{EnglHankeNeubauer:1996,ItoJin:2015}, and admits interpretation in the lens of convex duality.
Further, we derive error bounds on the approximation
with respect to the true solution $x^\dag$ in Section \ref{sec:error}, in the spirit of regularization theory for noisy operators.
These results provide guidelines on the low-rank approximation, and differ
from existing results \cite{BoutsidisMagdon:2014,JiaWang:2018,XiangZou:2013,XiangZou:2015,WeiXieZhang:2016},
where the focus is on relative error estimates with respect to the regularized solution.

Now we situate the work in the literature on RSVD for inverse problems. RSVD has been applied to solving
inverse problems efficiently \cite{BoutsidisMagdon:2014,XiangZou:2013,XiangZou:2015,WeiXieZhang:2016}.
Xiang and Zou \cite{XiangZou:2013} developed RSVD for standard Tikhonov regularization and provided relative
error estimates between the approximate and exact Tikhonov minimizer, by adapting the perturbation theory
for least-squares problems. In the work \cite{XiangZou:2015}, the authors proposed two approaches
based respectively on transformation to standard form and randomized generalized SVD (RGSVD), and for the
latter, RSVD is only performed on the matrix $A$. There was no error estimate in \cite{XiangZou:2015}.
Wei et al \cite{WeiXieZhang:2016} proposed different
implementations, and derived some relative error estimates. Boutsidis and Magdon \cite{BoutsidisMagdon:2014}
analyzed the relative error for truncated RSVD, and discussed the sample complexity.
Jia and Yang \cite{JiaWang:2018} presented a different way to perform truncated RSVD via
LSQR for general smooth penalty, and provided relative error estimates. See also \cite{KluthJin:2019} for
an evaluation within magnetic particle imaging. More generally,
the idea of randomization has been fruitfully employed to reduce the computational cost associated with
regularized inversion in statistics and machine learning, under the name of
sketching in either primal or dual spaces \cite{ChenLiuLyuKingZhang:2016,PilanciWainwright:2016,ZhangMahdaviJinYang:2014,
WangLeeMahdaviKolar:2017}. All these works also
essentially exploit the low-rank structure, but in a different manner. Our analysis may also be extended to these approaches.

The rest of the paper is organized as follows. In Section \ref{sec:prelim}, we recall preliminaries on
RSVD, especially implementation and error bound. Then in Section \ref{sec:reg}, under one
single guiding principle, we develop efficient inversion schemes based on RSVD for three classical regularization
methods, and give the error analysis in Section \ref{sec:error}. Finally we illustrate the approaches
with some numerical results in Section \ref{sec:numer}. In the appendix, we describe an iterative refinement
scheme for (general) Tikhonov regularization. Throughout, we denote by lower and capital letters for vectors and
matrices, respectively, by $I$ an identity matrix of an appropriate size, by $\|\cdot\|$ the Euclidean
norm for vectors and spectral norm for matrices,
and by $(\cdot,\cdot)$ for Euclidean inner product for vectors. The superscript $^*$ denotes
the vector/matrix transpose. We use the notation $\mathcal{R}(A)$ and $\mathcal{N}(A)$ to denote the range
and kernel of a matrix $A$, and $A_k$ and $\tilde A_k$ denote the optimal and approximate rank-$k$
approximations by SVD and RSVD, respectively. The notation $c$ denotes a generic constant which may
change at each occurrence, but is always independent of the condition number of $A$.

\section{Preliminaries}\label{sec:prelim}
Now we recall preliminaries on RSVD and technical lemmas.

\subsection{SVD and pseudoinverse}

Singular value decomposition (SVD) is one of most powerful tools in numerical linear algebra. For
any matrix $A\in\mathbb{R}^{n\times m}$, SVD of $A$ is given by
\begin{equation*}
  A=U\Sigma V^*,
\end{equation*}
where $U=[u_1\ u_2 \ \cdots \ u_n] \in \mathbb{R}^{n\times n}$ and $V=[v_1\ v_2 \ \cdots \ v_m]\in
\mathbb{R}^{m\times m}$ are column orthonormal matrices, with the vectors $u_i$ and $v_i$ being the
left and right singular vectors, respectively, and $V^*$ denotes the transpose of $V$. The diagonal
matrix $\Sigma=\mathrm{diag}(\sigma_i)\in\mathbb{R}^{n\times m}$
has nonnegative diagonal entries $\sigma_i$, known as singular values (SVs), ordered nonincreasingly:
\begin{equation*}
  \sigma_1\geq \sigma_2\geq \ldots\geq \sigma_r>\sigma_{r+1}=\ldots=\sigma_{\min(m,n)}=0,
\end{equation*}
where $r=\mathrm{rank}(A)$ is the rank of $A$. Let $\sigma_i(A)$ be the
$i$th SV of $A$.  The complexity of the standard
Golub-Reinsch algorithm for computing SVD is $4n^2m + 8m^2n + 9m^3$ (for $n\geq m$)
\cite[p. 254]{GolubvanLoan:1996}. Thus, it is expensive for large-scale problems.

Now we can give the optimal low-rank approximation to $A$. By Eckhardt-Young theorem, the optimal
rank-$k$ approximation $A_k$ of $A$ (in spectral norm) is given by
\begin{equation*}
  \|A-U_k\Sigma_kV^*_k\| = \sigma_{k+1},
\end{equation*}
where $U_k\in\mathbb{R}^{n\times k}$ and $V_k\in\mathbb{R}^{m\times k}$ are the submatrix formed by taking
the first $k$ columns of the matrices $U$ and $V$, and $\Sigma_k=\mathrm{diag}(\sigma_1,\ldots,\sigma_k)\in\mathbb{R}^{k\times k}$.
The pseudoinverse $A^\dag\in\mathbb{R}^{m\times n}$ of $A \in\mathbb{R}^{n\times m}$ is given by
\begin{equation*}
  A^\dag = V_r\Sigma_r^{-1}U_r^*.
\end{equation*}

We have the following properties of the pseudoinverse of matrix product.
\begin{lemma}\label{lem:pseudoinverse}
For any $A\in\mathbb{R}^{m\times n}$, $B\in\mathbb{R}^{n\times l}$, the
identity $(AB)^\dag = B^\dag A^\dag$ holds, if one of the following conditions is fulfilled:
{\rm(i)} $A$ has orthonormal columns;
{\rm(ii)} $B$ has orthonormal rows;
{\rm(iii)} $A$ has full column rank and $B$ has full row rank.
\end{lemma}

The next result gives an estimate on matrix pseudoinverse.
\begin{lemma}\label{lem:pseudoinverse2}
For symmetric semipositive definite $A,B\in\mathbb{R}^{m\times m}$, there holds
\begin{equation*}
 \| A^\dag - B^\dag\| \leq \|A^\dag\|\|B^\dag\|\|B-A\|.
\end{equation*}
\end{lemma}
\begin{proof}
Since $A$ is symmetric semipositive definite, we have
$A^\dag = \lim_{\mu\to 0^+}(A+\mu I)^{-1}.$ By the identity $C^{-1}-D^{-1}=C^{-1}
(D-C)D^{-1}$ for invertible $C,D\in\mathbb{R}^{m\times m}$,
\begin{equation*}
 \begin{aligned}
   A^\dag - B^\dag & = \lim_{\mu\to0^+}[(A+\mu I)^{-1}-(B+\mu I)^{-1}] \\
   &= \lim_{\mu\to0^+}[(A+\mu I)^{-1}(B-A)(B+\mu I)^{-1}]= A^\dag(B-A)B^\dag.
 \end{aligned}
\end{equation*}
Now the estimate follows from the matrix spectral norm estimate.\qed
\end{proof}

\begin{remark}
The estimate for general matrices is weaker than the one in Lemma \ref{lem:pseudoinverse2}:
for general $A,B\in\mathbb{R}^{n\times m}$ with $\mathrm{rank}(A)=\mathrm{rank}(B)<\min(m,n)$, there holds \cite{Stewart:1977}
\begin{equation*}
  \|A^\dag - B^\dag\|\leq \tfrac{1+\sqrt{5}}{2}\|A^\dag\|\|B^\dag\|\|B-A\|.
\end{equation*}
The rank condition is essential, and otherwise, the estimate may not hold.
\end{remark}

Last, we recall the stability of SVs (\cite[Cor. 7.3.8]{HornJohnson:1985}, \cite[Sec. 1.3]{Tao:2012}).
\begin{lemma}\label{lem:Weyl}
For $A,B\in\mathbb{R}^{n\times m}$, there holds
\begin{equation*}
  |\sigma_i(A+B)-\sigma_i(A)|\leq \|B\|,\quad i=1,\ldots,\min(m,n).
\end{equation*}
\end{lemma}

\subsection{Randomized SVD}

Traditional numerical methods to compute a rank-$k$ SVD, e.g., Lanczos bidiagonalization and Krylov
subspace method, are especially powerful for large sparse or structured matrices. However, for
many discrete inverse problems, there is no such structure. The prototypical model in inverse
problems is a Fredholm integral equation of the first kind, which gives rise to unstructured dense matrices. Over the past decade, randomized algorithms
for computing low-rank approximations have gained popularity. Frieze
\textit{et al} \cite{FriezeKannanVempala:2004} developed a Monte Carlo SVD to efficiently compute an approximate low-rank SVD
based on non-uniform row and column sampling. Sarlos \cite{Sarlos:2006} proposed an approach based on random projection,
using properties of random vectors to build a subspace capturing the matrix range. Below we describe
briefly the basic idea of RSVD, and refer readers to \cite{HalkoMartinssonTropp:2011} for an overview
and to \cite{Gu:2015,MuscoMusco:2015,SzlamTulloch:2017} for an incomplete list of recent works.

RSVD can be viewed as an iterative procedure based on SVDs of a sequence of low-rank matrices to deliver a nearly optimal low-rank SVD.
Given a matrix $A\in\mathbb{R}^{n\times m}$ with $n\geq m$, we aim at obtaining a rank-$k$ approximation, with
$k\ll\min( m,n)$. Let $\Omega \in\mathbb{R}^{m\times(k+p)}$, with $k+p\leq m$, be a random matrix, with its
entries following an i.i.d. Gaussian distribution $N(0,1)$, and the integer
$p\geq 0$ is an oversampling parameter (with a default value $p=5$ \cite{HalkoMartinssonTropp:2011}).
Then we form a random matrix $Y$ by
\begin{equation}\label{eqn:Y}
  Y = (AA^*)^qA\Omega,
\end{equation}
where the exponent $q\in\mathbb{N}\cup\{0\}$. By SVD of $A$, i.e., $A=U\Sigma V^*$, $Y$ is given by
\begin{equation*}
   Y = U\Sigma^{2q+1}V^*\Omega.
\end{equation*}
Thus $\Omega$ is used for probing $\mathcal{R}(A)$, and $\mathcal{R}(Y)$ captures $\mathcal{R}(U_k)$
well. The accuracy is determined by the decay of $\sigma_i$s, and the exponent $q$ can greatly
improve the performance when $\sigma_i$s decay slowly. Let $Q\in\mathbb{R}^{n\times (k+p)}$ be an orthonormal basis
for $\mathcal{R}(Y)$, which can be computed efficiently via QR
factorization or skinny SVD. Next we form the (projected) matrix
\begin{equation*}
  B = Q^*A \in \mathbb{R}^{(k+p)\times m}.
\end{equation*}
Last, we compute SVD of $B$
\begin{equation*}
  B= WSV^*,
\end{equation*}
with $W\in \mathbb{R}^{(k+p)\times (k+p)}$, $S\in\mathbb{R}^{(k+p)\times (k+p)}$ and $V\in \mathbb{R}^{m\times (k+p)}$.
This again can be carried out efficiently by standard SVD, since the size of $B$ is much smaller. With $1:k$ denoting the
index set $\{1,\ldots,k\}$, let
$\tilde U_k = QW(1:n,1:k)\in\mathbb{R}^{n\times k}$, $\tilde \Sigma_k = S(1:k,1:k)\in \mathbb{R}^{k\times k}$
and $\tilde V_k =V(1:m,1:k)\in \mathbb{R}^{m\times k}$.
The triple $(\tilde U_k,\tilde\Sigma_k,\tilde V_k)$ defines a rank-$k$ approximation $\tilde A_k$:
\begin{equation*}
  \tilde A_k = \tilde U_k\tilde \Sigma_k \tilde V^*_k.
\end{equation*}
The triple $(\tilde U_k,\tilde \Sigma_k,\tilde V_k)$ is a nearly optimal rank-$k$ approximation to
$A$; see Theorem \ref{thm:rsvd} below for a precise statement. The approximation is random due to range probing
by $\Omega$. By its very construction, we have
\begin{equation}\label{eqn:rsvd}
  \tilde A_k = \tilde P_k A,
\end{equation}
where $\tilde P_k=\tilde U_k\tilde U_k^*\in\mathbb{R}^{n\times n}$ is the orthogonal projection
into $\mathcal{R}(\tilde U_k)$. The procedure for RSVD is given in Algorithm \ref{alg:rsvd}.
The  complexity of Algorithm \ref{alg:rsvd} is about $4(q+1)nmk$, which can be much smaller than
that of full SVD if $k\ll \min(m,n)$.

\begin{algorithm}[hbt!]
  \centering
  \caption{RSVD for $A\in\mathbb{R}^{n\times m}$, $n\geq m$.\label{alg:rsvd}}
  \begin{algorithmic}[1]
    \STATE Input matrix $A\in\mathbb{R}^{n\times m}$, $n\geq m$, and target rank $k$;
    \STATE Set parameters $p$ (default $p=5$), and $q$ (default $q=0$);
    \STATE Sample a random matrix $\Omega=(\omega_{ij})\in\mathbb{R}^{m\times (k+p)}$, with $\omega_{ij}\sim N(0,1)$;
    \STATE Compute the randomized matrix $Y=(AA^*)^qA\Omega$;
    \STATE Find an orthonormal basis $Q$ of $\mathrm{range}(Y)$ by QR decomposition;
    \STATE Form the matrix $B=Q^*A$;
    \STATE Compute the SVD of $B=WSV^*$;
    \STATE Return the rank $k$ approximation $(\tilde U_k,\tilde \Sigma_k,\tilde V_k)$, cf. \eqref{eqn:rsvd}.
  \end{algorithmic}
\end{algorithm}

\begin{remark}
The SV $\sigma_i$ can be characterized by \cite[Theorem 8.6.1, p. 441]{GolubvanLoan:1996}:
\begin{equation*}
  \sigma_i = \max_{\substack{u\in\mathbb{R}^n, u\perp\mathrm{span}(\{u_j\}_{j=1}^{i-1})}}\frac{\|A^*u\|}{\|u\|}.
\end{equation*}
Thus, one may estimate $\sigma_i(A)$ directly by $\tilde\sigma_i(A) = \|A^* \tilde U(:,i)\|$, and
refine the SV estimate, similar to Rayleigh quotient
acceleration for computing eigenvalues.
\end{remark}

The following error estimates hold for RSVD $(\tilde U_k,\tilde \Sigma_k,\tilde V_k)$ given
by Algorithm \ref{alg:rsvd} with $q=0$ \cite[Cor. 10.9, p. 275]{HalkoMartinssonTropp:2011},
where the second estimate shows how the parameter $p$ improves the accuracy. The exponent $q$
is in the spirit of a power method, and can significantly improve the accuracy in the absence
of spectral gap; see \cite[Cor. 10.10, p. 277]{HalkoMartinssonTropp:2011} for related discussions.

\begin{theorem}\label{thm:rsvd}
For $A\in\mathbb{R}^{n\times m}$, $n\geq m$, let $\Omega\in\mathbb{R}^{m\times (k+p)}$ be a standard
Gaussian matrix, $k+p\leq m$ and $p\geq 4$, and $Q$ an orthonormal basis for $\mathcal{R}(A\Omega)$. Then with probability at least $1-3p^{-p}$, there holds
\begin{equation*}
  \|A-QQ^*A\|\leq (1+6((k+p)p\log p)^\frac{1}{2})\sigma_{k+1}+3\sqrt{k+p}\Big(\sum_{j>k}\sigma_j^2\Big)^\frac{1}{2},
\end{equation*}
and further with probability at least $1-3e^{-p}$, there holds
\begin{equation*}
  \|A-QQ^*A\|\leq \Big(1+16\Big(1+\frac{k}{p+1}\Big)^\frac{1}{2}\Big)\sigma_{k+1} + \frac{8(k+p)^\frac{1}{2}}{p+1}\Big(\sum_{j>k}\sigma_j^2\Big)^\frac{1}{2}.
\end{equation*}
\end{theorem}

The next result is an immediate corollary of Theorem \ref{thm:rsvd}. Exponentially
decaying SVs arise in, e.g., backward heat conduction and elliptic Cauchy problem.
\begin{corollary}
Suppose that the SVs $\sigma_i$ decay exponentially, i.e., $\sigma_j=c_0c_1^{j}$, for
some $c_0>0$ and $c_1\in(0,1)$. Then with probability at least $1-3p^{-p}$, there holds
\begin{equation*}
  \|A-QQ^*A\|\leq \Big[1+6((k+p)p\log p)^\frac{1}{2}+\frac{3(k+p)^\frac12}{(1-c_1^2)^\frac12}\Big]\sigma_{k+1},
\end{equation*}
and further with probability at least $1-3e^{-p}$, there holds
\begin{equation*}
  \|A-QQ^*A\|\leq \Big[\Big(1+16\Big(1+\frac{k}{p+1}\Big)^\frac{1}{2}\Big) + \frac{8(k+p)^\frac{1}{2}}{(p+1)(1-c_1^2)^{\frac12}}\Big]\sigma_{k+1}.
\end{equation*}
\end{corollary}

So far we have assumed that $A$ is tall, i.e., $n\geq m$. For the
case $n<m$, one may apply RSVD to $A^*$, which gives
rise to Algorithm \ref{alg:rsvd*}.

\begin{algorithm}
  \centering
  \caption{RSVD for $A\in\mathbb{R}^{n\times m},\ n<m$.\label{alg:rsvd*}}
  \begin{algorithmic}[1]
    \STATE Input matrix $A\in\mathbb{R}^{n\times m}$, $n<m$, and target rank $k$;
    \STATE Set parameters $p$ (default $p=5$), and $q$ (default $q=0$);
    \STATE Sample a random matrix $\Omega=(\omega_{ij})\in\mathbb{R}^{(k+p)\times n}$, with $\omega_{ij}\sim N(0,1)$;
    \STATE Compute the randomized matrix $Y=\Omega A(A^*A)^q$;
    \STATE Find an orthonormal basis $Q$ of $\mathrm{range}(Y^*)$ by QR decomposition;
    \STATE Find the matrix $B=AQ$;
    \STATE Compute the SVD of $B=USV^*$;
    \STATE Return the rank $k$ approximation $(\tilde U_k,\tilde \Sigma_k,\tilde V_k)$.
  \end{algorithmic}
\end{algorithm}

The efficiency of RSVD resides crucially on the truly low-rank nature of the problem.
The precise spectral decay is generally unknown for many practical inverse problems,
although there are known estimates for several model problems, e.g., X-ray transform \cite{Maass:1987}
and magnetic particle imaging \cite{KluthJinLi:2018}. The decay rates generally worsen with the increase of
the spatial dimension $d$, at least for integral operators \cite{GriebelLi:2018}, which can potentially
hinder the application of RSVD type techniques to high-dimensional problems.

\section{Efficient regularized linear inversion with RSVD}\label{sec:reg}

Now we develop efficient inversion techniques based on RSVD for problem \eqref{eqn:lininv} via truncated SVD (TSVD), Tikhonov
regularization and Tikhonov regularization with a smoothness penalty \cite{EnglHankeNeubauer:1996,ItoJin:2015}.
For large-scale inverse problems, this can  be expensive, since they either involve full SVD or large dense linear systems.
We aim at reducing the cost by exploiting the inherent low-rank structure for inverse problems,
and accurately constructing a low-rank approximation by RSVD. This idea has been pursued recently
\cite{BoutsidisMagdon:2014,JiaWang:2018,XiangZou:2013,XiangZou:2015,WeiXieZhang:2016}. Our work is along the same line in
of research but with a unified framework for deriving all three approaches and interpreting the approach in the lens of convex duality.

The key observation is the range type condition on the approximation $\tilde x$:
\begin{equation}\label{eqn:range}
  \tilde x \in\mathcal{R}(B),
\end{equation}
with the matrix $B$ is given by
\begin{equation*}
 B =\left\{\begin{array}{ll}
    A^*, & \text{truncated SVD, Tikhonov},\\
    L^\dag L^{*\dag}A^*, & \text{general Tikhonov},
  \end{array}\right.
\end{equation*}
where $L$ is a regularizing matrix, typically chosen to the finite difference approximation of the
first- or high-order derivatives \cite{EnglHankeNeubauer:1996}. Similar to \eqref{eqn:range},
the approximation $\tilde x$ is assumed to live in $\mathrm{span}(\{v_i\}_{i=1}^k)$ in
\cite{ZhangMahdaviJinYang:2014} for Tikhonov regularization, which is
slightly more restrictive than \eqref{eqn:range}.
An analogous condition on the exact solution $x^\dag$ reads
\begin{equation}\label{eqn:range-ex}
 x^\dag = B w
\end{equation}
for some $w\in\mathbb{R}^n$. In regularization theory \cite{EnglHankeNeubauer:1996,ItoJin:2015},
\eqref{eqn:range-ex} is known as source condition, and can be viewed as the Lagrange multiplier for the
equality constraint $Ax^\dag=b^\dag$, whose existence is generally not ensured for infinite-dimensional problems.
It is often employed to bound the error $\|\tilde x-x^\dag\|$ of the approximation $\tilde x$ in terms
of the noise level $\delta$. The construction below explicitly maintains \eqref{eqn:range},
thus preserving the structure of the regularized solution $\tilde x$. We will interpret the construction by
convex analysis. Below we develop three efficient computational schemes based on RSVD.

\subsection{Truncated RSVD}\label{ssec:tsvd}

Classical truncated SVD (TSVD) stabilizes problem \eqref{eqn:lininv} by looking for the least-squares solution of
\begin{equation*}
  \min \|A_k x_k - b\|,\quad \quad \mbox{with } A_k =U_k\Sigma_k V_k^*.
\end{equation*}
Then the regularized solution $x_k $ is given by
\begin{equation*}
  x_k = A_k^\dag b = V_k \Sigma_k^{-1}U_k^* b=\sum_{i=1}^k\sigma_i^{-1}(u_i,b)v_i.
\end{equation*}
The truncated level $k\leq \mathrm{rank}(A)$ plays the role of a regularization parameter,
and determines the strength of regularization. TSVD requires computing
the (partial) SVD of $A$, which is expensive for large-scale problems. Thus, one can substitute a rank-$k$ RSVD
$(\tilde U_k,\tilde \Sigma_k,\tilde V_k)$, leading to truncated RSVD (TRSVD):
\begin{equation*}
  \hat x_k = \tilde V_k \tilde \Sigma_k^{-1}\tilde U_k^*b.
\end{equation*}
By Lemma \ref{lem:Weyl}, $\tilde A_k=\tilde U_k\tilde \Sigma_k\tilde V_k^*$ is indeed of rank $k$, if
$\|A-\tilde A_k\|<\sigma_k$. This approach was adopted in \cite{BoutsidisMagdon:2014}. Based on RSVD, we
propose an approximation $\tilde x_k$ defined by
\begin{equation}\label{eqn:TrSVD}
  \tilde x_k = A^*(\tilde A_k\tilde A_k^*)^{\dag}b = A^*\sum_{i=1}^k\frac{(\tilde u_i,b)}{\tilde \sigma_i^2}\tilde u_i.
\end{equation}
By its construction, the range condition \eqref{eqn:range} holds for $\tilde x_k$.
To compute $\tilde x_k$, one does not need the complete RSVD $(\tilde U_k,\tilde\Sigma_k,\tilde V_k)$
of rank $k$, but only $(\tilde U_k,\tilde \Sigma_k)$, which is advantageous for
complexity reduction \cite[p. 254]{GolubvanLoan:1996}. Given the RSVD
$(\tilde U_k,\tilde \Sigma_k)$, computing $\tilde x_k$ by \eqref{eqn:TrSVD} incurs
only $O(nk+nm)$ operations.

\subsection{Tikhonov regularization}\label{ssec:t-Tikh}

Tikhonov regularization stabilizes \eqref{eqn:lininv} by minimizing the following functional
\begin{equation*}
 J_\alpha (x) = \tfrac{1}{2}\|Ax - b\|^2 + \tfrac{\alpha}{2}\|x\|^2,
\end{equation*}
where $\alpha>0$ is the regularization parameter. The regularized solution $x_\alpha$ is given by
\begin{equation}\label{eqn:x_tikh}
  x_\alpha = (A^*A+\alpha I)^{-1}A^*b = A^*(AA^*+\alpha I)^{-1}b.
\end{equation}
The latter identity verifies \eqref{eqn:range}. The cost of the step in \eqref{eqn:x_tikh} is about
$nm^2+\frac{m^3}{3}$ or $mn^2+\frac{n^3}{3}$ \cite[p. 238]{GolubvanLoan:1996}, and thus it is
expensive for large scale problems. One approach to accelerate the computation is to apply the RSVD
approximation $\tilde A_k =\tilde U_k\tilde\Sigma_k\tilde V_k^*$. Then one
obtains a regularized approximation \cite{XiangZou:2013}
\begin{equation}\label{eqn:t-tikh_x-xz}
  \hat x_\alpha = (\tilde A_k^*\tilde A_k+\alpha I)^{-1}\tilde A_k^*b.
\end{equation}
To preserve the range property \eqref{eqn:range}, we propose an alternative
\begin{equation}\label{eqn:t-tikh_x}
  \tilde x_\alpha = A^*(\tilde A_k\tilde A_k^*+\alpha I)^{-1}b=A^* \sum_{i=1}^k \frac{(\tilde u_i,b)}{\tilde \sigma_i^2+\alpha}\tilde u_i.
\end{equation}
For $\alpha\to0^+$, $\tilde x_\alpha$ recovers the TRSVD $\tilde x_k$ in \eqref{eqn:TrSVD}.
Given RSVD $(\tilde U_k,\tilde \Sigma_k)$, the complexity of computing $\tilde x_\alpha $
is nearly identical with the TRSVD $\tilde x_k$.

\subsection{General Tikhonov regularization}\label{ssec:t-gTikh}

Now we consider Tikhonov regularization with a general smoothness penalty:
\begin{equation}\label{eqn:gTikh}
  J_\alpha(x) = \tfrac{1}{2}\|Ax-b\|^2 + \tfrac{\alpha}{2}\|Lx\|^2,
\end{equation}
where $L\in\mathbb{R}^{\ell\times m}$ is a regularizing matrix enforcing smoothness. Typical choices
of $L$ include first-order and second-order derivatives. We
assume $\mathcal{N}(A)\cap \mathcal{N}(L)=\{0\}$ so that $J_\alpha$ has a unique minimizer $x_\alpha$.
By the identity
\begin{equation}\label{eqn:matrix-id}
  (A^*A+\alpha I)^{-1}A^* = A^*(AA^*+\alpha I)^{-1},
\end{equation}
if $\mathcal{N}(L)=\{0\}$, the minimizer $x_\alpha$ to $J_\alpha$ is given by (with $\Gamma=L^{\dag}L^{\dag*}$)
\begin{align}
    x_\alpha &= (A^*A+\alpha L^*L)^{-1}(A^*y) \nonumber\\
    &= L^{\dag}((AL^{\dag})^* AL^{\dag} + \alpha I)^{-1}(AL^{\dag})^*b\nonumber \\
    &= \Gamma A^*(A\Gamma A^* +\alpha I)^{-1}b.\label{eqn:solrep-gen}
\end{align}
The $\Gamma$ factor reflects the smoothing property of $\|Lx\|^2$.
Similar to \eqref{eqn:t-tikh_x}, we approximate $B:=AL^\dag$ via RSVD:
$ \tilde B_k= U_k\Sigma_kV_k^*$,
and obtain a regularized solution $\tilde x_\alpha$ by
\begin{equation}\label{eqn:t-tikh_x-gen}
  \tilde x_\alpha = \Gamma A^*(\tilde B_k \tilde B_k^* +\alpha I)^{-1}b.
\end{equation}
 It differs from
\cite{XiangZou:2015} in that \cite{XiangZou:2015} uses only the RSVD approximation
of $A$, thus it does not maintain the range condition \eqref{eqn:range-ex-gen}.
The first step of Algorithm \ref{alg:rsvd}, i.e., $AL^{-1}\Omega$, is to
probe $\mathcal{R}(A)$ with colored Gaussian noise with covariance $\Gamma$.

Numerically, it also involves applying $\Gamma$, which can be carried out efficiently
if $L$ is structured. If $L$ is rectangular, we have the following decomposition \cite{Elden:1982,
XiangZou:2013}. The $A$-weighted pseudoinverse $L^\#$ \cite{Elden:1982}
can be computed efficiently, if $L^\dag$ is easy to compute and the dimensionality of $W$ is small.
\begin{lemma}\label{lem:solrep-gtikh}
Let $W$ and $Z$ be any matrices satisfying $\mathcal{R}(W)=\mathcal{N}(L)$, $\mathcal{R}(Z)=\mathcal{R}(L)$, $Z^*Z=I$,
and $L^\#=(I-W(AW)^\dag A)L^\dag$. Then the solution $x_\alpha$ to \eqref{eqn:gTikh} is given by
\begin{equation}\label{eqn:x-al-ZW}
  x_\alpha = L^\#Z\xi_\alpha + W(AW)^\dag b,
\end{equation}
where the variable $\xi_\alpha$ minimizes $\frac12\|AL^\#Z\xi-b\|^2 + \frac\alpha2 \|\xi\|^2$.
\end{lemma}

Lemma \ref{lem:solrep-gtikh} does not necessarily entail an efficient scheme, since it requires
an orthonormal basis $Z$ for $\mathcal{R}(L)$. Hence, we restrict our discussion to the case:
\begin{equation}\label{eqn:cond-L}
   L\in\mathbb{R}^{\ell\times m} \quad\mbox{ with } \mathrm{rank}(L)=\ell<m.
\end{equation}
It arises most commonly in practice, e.g., first-order or second-order derivative, and there
are efficient ways to perform standard-form reduction. 
Then we can let $Z=I_\ell$. By slightly abusing the notation $\Gamma = L^\#L^{\#*}$, by
Lemma \ref{lem:solrep-gtikh}, we have
\begin{equation*}
  \begin{aligned}
    x_\alpha &= L^\#((AL^\#)^*AL^\#+\alpha I)^{-1}(AL^\#)^*b + W(AW)^\dag b\\
      & = \Gamma A^*(A\Gamma A^*+\alpha I)^{-1}b + W(AW)^\dag b.
  \end{aligned}
\end{equation*}
The first term is nearly identical with \eqref{eqn:solrep-gen}, with $L^\#$ in place of $L^\dag$, and the
extra term $W(AW)^\dag b$ belongs to $\mathcal{N}(L)$. Thus, we obtain an approximation $\tilde x_\alpha$ defined by
\begin{equation}\label{eqn:t-tikh_x-gen2}
  \tilde x_\alpha = \Gamma A^* (\tilde B_k\tilde B_k + \alpha I)^{-1}b + W(AW)^\dag b,
\end{equation}
where $\tilde B_k$ is a rank-$k$ RSVD to $B\equiv AL^\#$. The matrix $B$ can be implemented implicitly
via matrix-vector product to maintain the efficiency.

\subsection{Dual interpretation}
Now we give an interpretation of \eqref{eqn:t-tikh_x-gen} in the lens of Fenchel duality theory in Banach
spaces (see, e.g., \cite[Chapter II.4]{EkelandTemam:1999}). Recall that for a functional $F:X\to
\overline{\mathbb{R}}:=\mathbb{R}\cup\{\infty\}$ defined on a Banach space $X$, let
$F^*:X^*\to \overline{\mathbb{R}}$ denote the Fenchel conjugate of $F$ given for $x^*\in X^*$ by
\begin{equation*}
  F^*(x^*) = \sup_{x\in X} \langle x^*,x\rangle_{X^*,X}-F(x).
\end{equation*}
Further, let
$\partial F(x):=\{x^*\in X^*: \langle x^*,\tilde x-x\rangle_{X^*,X}\leq F(\tilde x)-F(x)\ \ \forall \tilde x\in X\}$
be the subdifferential of the convex functional $F$ at $x$, which coincides with G\^{a}teaux
derivative $F'(x)$ if it exists. The Fenchel duality theorem states that if $F:X\to \overline{
\mathbb{R}}$ and $G:Y\to\overline{\mathbb{R}}$ are proper, convex and lower semicontinuous
functionals on the Banach spaces $X$ and $Y$, $\Lambda:X\to Y$ is a continuous linear operator,
and there exists an $x_0\in W$ such that $F(x_0)<\infty$, $G(\Lambda x_0)<\infty$, and $G$ is
continuous at $\Lambda x_0$, then
\begin{equation*}
  \inf_{x\in X} F(x)+G(\Lambda x) = \sup_{y^*\in Y^*} -F^*(\Lambda^*y^*)-G^*(-y^*),
\end{equation*}
Further, the equality is attained at $(\bar x, \bar y^*)\in X\times Y^*$ if and only if
\begin{align}\label{eqn:Fenchel}
    \Lambda^* \bar y^* \in \partial F(\bar x)\quad\mbox{and} \quad
    - \bar y^* \in \partial G(\Lambda \bar x),
\end{align}
hold \cite[Remark III.4.2]{EkelandTemam:1999}.

The next result indicates that the approach in Sections \ref{ssec:t-Tikh}-\ref{ssec:t-gTikh} first applies RSVD
to the dual problem to obtain an approximate dual $\tilde p_\alpha$, and then recovers the optimal primal $\tilde x_\alpha$ via
duality relation \eqref{eqn:Fenchel}. This connection is in the same spirit of dual random projection
\cite{WangLeeMahdaviKolar:2017,ZhangMahdaviJinYang:2014}, and it opens up the avenue to extend RSVD to
functionals whose conjugate is simple, e.g., nonsmooth fidelity.

\begin{prop}\label{prop:dual}
If $\mathcal{N}(L)=\{0\}$, then $\tilde x_\alpha$ in \eqref{eqn:t-tikh_x-gen} is
equivalent to RSVD for the dual problem.
\end{prop}
\begin{proof}
For any symmetric positive semidefinite $Q$, the conjugate functional
$F^*$ of $F(x) = \frac{\alpha}{2}x^*Qx$ is given by $F^*(\xi) =-\frac{1}{2\alpha}\xi^*Q^\dag \xi$, with
its domain being $\mathcal{R}(Q)$. By SVD, we have $(L^*L)^\dag=L^\dag L^{\dag*}$, and
thus $F^*(\xi)=-\frac{1}{2\alpha}\|L^{\dag*}\xi\|^2$. Hence, by Fenchel duality theorem, the conjugate
$J_\alpha^*(\xi)$ of $J_\alpha(x)$ is given by
\begin{equation*}
  J^*_\alpha (\xi) : = -\tfrac{1}{2\alpha}\|L^{\dag*} A^* \xi\|^2 - \tfrac{1}{2}\|\xi-b\|^2.
\end{equation*}
Further, by \eqref{eqn:Fenchel}, the optimal primal and dual pair $(x_\alpha,\xi_\alpha)$ satisfies
\begin{equation*}
 \alpha L^*L x_\alpha = A^*\xi_\alpha\quad \mbox{and}\quad \xi_\alpha = b-Ax_\alpha.
\end{equation*}
Since $\mathcal{N}(L)=\{0\}$, $L^*L$ is invertible, and thus $x_\alpha = \alpha^{-1} (L^*L)^{-1}
A^*\xi_\alpha = \alpha^{-1} \Gamma A^*\xi_\alpha$. The optimal dual $\xi_\alpha$ is given by
$\xi_\alpha = \alpha (AL^{\dag}L^{*\dag} A^*+\alpha I)^{-1}b$. To approximate $\xi_\alpha$ by
$\tilde \xi_\alpha$, we employ the RSVD approximation $\tilde B_k$ to $B=AL^\dag$ and solve
\begin{equation*}
  \tilde \xi_\alpha =\arg\max_{\xi\in\mathbb{R}^n} \{-\tfrac{1}{2\alpha}\|\tilde B_k^*
  \xi\|^2 - \tfrac12\|\xi-b\|^2\}.
\end{equation*}
We obtain an approximation via the relation
$\tilde x_\alpha = \alpha^{-1}\Gamma A^*\tilde \xi_\alpha$, recovering \eqref{eqn:t-tikh_x-gen}.\qed
\end{proof}

\begin{remark}
For a general regularizing matrix $L$, one can appeal to the decomposition in Lemma \ref{lem:solrep-gtikh},
by applying first the standard transformation and then approximating the regularized part
via convex duality.
\end{remark}

\section{Error analysis}\label{sec:error}

Now we derive error estimates for the approximation $\tilde x$ with respect to the true solution $x^\dag$, under
sourcewise type conditions. In addition to bounding the error, the estimates provide
useful guidelines on constructing the approximation $\tilde A_k$.

\subsection{Truncated RSVD}
We derive an error estimate under the source condition \eqref{eqn:range-ex}.
We use the projection matrices $P_k=U_kU_k^*$ and $\tilde P_k=
\tilde U_k\tilde U_k^*$ frequently below.

\begin{lemma}\label{lem:est-AA_k}
For any $k\leq r$ and $\|A-\tilde A_k\|\leq \sigma_k/2$, there holds
$\|A^* (\tilde A_k^*)^\dag\|\leq 2.$
\end{lemma}
\begin{proof}
It follows from the decomposition $A=\tilde P_k A + (I-\tilde P_k)A=\tilde A_k + (I-\tilde P_k)A$ that
\begin{equation*}
  \begin{aligned}
    \|A^*(\tilde A_k^*)^\dag\| & = \|(\tilde A_k+(I-\tilde P_k)A)^*(\tilde A_k^*)^\dag\|\leq \|\tilde A_k^*(\tilde A_k^*)^{-1}\| + \|A-\tilde A_k\|\|\tilde A_k^{-1}\|\\
      & \leq 1 + \tilde\sigma_k^{-1}\|A-\tilde A_k\|.
  \end{aligned}
\end{equation*}
Now the condition $\|A-\tilde A_k\|\leq \sigma_k/2$ and Lemma \ref{lem:Weyl} imply
$\tilde\sigma_k\geq \sigma_k-\|A-\tilde A_k\|\geq \sigma_k/2$, from which the desired estimate follows.\qed
\end{proof}

Now we can state an error estimate for the approximation $\tilde x_k$.
\begin{theorem}\label{thm:err-trsvd}
If Condition \eqref{eqn:range-ex} holds and $\|A-\tilde A_k\|\leq \sigma_k/2$, then for the  estimate
$\tilde x_k$ in \eqref{eqn:TrSVD}, there holds
\begin{equation*}
  \|x^\dag - \tilde x_k\|\leq 4\delta\sigma_k^{-1} + 8\sigma_1\sigma_k^{-1}\|A_k-\tilde A_k\|\|w\| +\sigma_{k+1}\|w\|.
\end{equation*}
\end{theorem}
\begin{proof}
By the decomposition $b=b^\dag + e$, we have (with $P_k^\perp=I-P_k$)
\begin{equation*}
  \begin{aligned}
    \tilde x_k -x^\dag & = A^*(\tilde A_k\tilde A_k^*)^\dag b- A^*(AA^*)^\dag b^\dag \\
      & = A^*(\tilde A_k\tilde A_k^*)^{\dag}e + A^*[(\tilde A_k\tilde A_k^*)^{\dag}-(A_kA_k^*)^{\dag}]b^\dag - P_k^\perp A^* (AA^*)^{\dag}b^\dag.
  \end{aligned}
\end{equation*}
The source condition $x^\dag = A^*w$ in \eqref{eqn:range-ex} implies
\begin{equation*}
 \tilde x_k -x^\dag = A^*(\tilde A_k\tilde A_k^*)^{\dag}e + A^*[(\tilde A_k\tilde A_k^*)^{\dag}-(A_kA_k^*)^{\dag}]AA^*w - P_k^\perp A^* (AA^*)^{\dag}AA^*w.
\end{equation*}
By the triangle inequality, we have
\begin{equation*}
  \begin{aligned}
  \|\tilde x_k-x^\dag \| & \leq \|A^*(\tilde A_k\tilde A_k^*)^{\dag}e\|+\|A^*[(\tilde A_k\tilde A_k^*)^{\dag}-(A_kA_k^*)^{\dag}]AA^*w\|\\
    &\quad +\|P_k^\perp A^* (AA^*)^{\dag}AA^*w\| :={\rm I}_1 + {\rm I}_2 + {\rm I}_3.
  \end{aligned}
\end{equation*}
It suffices to bound the three terms separately. First, for the term ${\rm I}_1$, by the identity $(\tilde A_k\tilde A_k^*)^\dag = (\tilde
A_k^*)^\dag\tilde A_k^\dag$ and Lemma \ref{lem:est-AA_k}, we have
\begin{equation*}
  {\rm I}_1 \leq \|A^*(\tilde A_k^*)^\dag\|\|\tilde A_k^\dag\|\|e\|\leq 2\tilde\sigma_k^{-1}\delta.
\end{equation*}
Second, for ${\rm I}_2$, by Lemmas \ref{lem:est-AA_k} and
\ref{lem:pseudoinverse2} and the identity $(\tilde A_k\tilde A_k^*)^\dag = (\tilde
A_k^*)^\dag\tilde A_k^\dag$, we have
\begin{equation*}
  \begin{aligned}
  {\rm I}_2 & \leq \|A^*[(\tilde A_k\tilde A_k^*)^{\dag}-(A_kA_k^*)^{\dag}]AA^*\|\|w\|\\
    & \leq \|A^*(\tilde A_k\tilde A_k^*)^\dag(\tilde A_k \tilde A_k^*-A_kA_k^*)(A_kA_k^*)^{\dag}AA^*\| \|w\|\\
     & \leq \|A^*(\tilde A_k^*) ^\dag\|\|\tilde A_k^\dag\|\|\tilde A_k\tilde A_k^*-A_kA_k^*\|\|(A_kA_k^*)^\dag AA^*\|\|w\|\\
     &\leq 4\tilde\sigma_k^{-1}\|A\|\|A_k-\tilde A_k\|\|w\|,
  \end{aligned}
\end{equation*}
since $\|\tilde A_k\tilde A_k^*-A_kA_k^*\|\leq \|\tilde A_k-A_k\|(\|\tilde A_k\|+\|A_k^*\|)\leq 2\|A\|\|\tilde A_k-A_k\|$
and $\|(A_kA_k^*)^\dag AA^*\|\leq 1$. By Lemma \ref{lem:Weyl}, we can bound the term $\|(\tilde A_k^*)^\dag\|$ by
\begin{equation*}
  \|(\tilde A_k^*)^\dag\| = \tilde \sigma_k^{-1}\leq (\sigma_k-\|A-\tilde A_k\|)^{-1}\leq 2\sigma_k^{-1}.
\end{equation*}
Last, we can bound the third term ${\rm I}_3$ directly by
${\rm I}_3 \leq \|P_k^\perp A^*\|\|w\| \leq \sigma_{k+1}\|w\|$.
Combining these estimates yields the desired assertion.\qed
\end{proof}

\begin{remark}
The bound in Theorem \ref{thm:err-trsvd} contains three terms: propagation error $\sigma_k^{-1}\delta$,
approximation error $\sigma_{k+1}\|w\|$, and perturbation error $\sigma_k^{-1}\|A\|\|A-\tilde A_k\|\|w\|$.
It is of the worst-case scenario type and can be pessimistic. In particular, the error
$\|A^*(\tilde A_k\tilde A_k^*)^{-1}e\|$ can be bounded more precisely by
\begin{equation*}
 \|A^*(\tilde A_k\tilde A_k^*)^{\dag}e\| \leq \|A^*(\tilde A_k^*)^\dag\|\|\tilde A_k^\dag e\|,
\end{equation*}
and $\|\tilde A_k^\dag e\|$ can be much smaller than $\tilde\sigma_k^{-1}\|e\|$, if $e$
concentrates in the high-frequency modes. By balancing the terms, it suffices for $\tilde A_k$ to have
an accuracy $O(\delta)$. This is consistent with the analysis for regularized solutions with perturbed operators.
\end{remark}

\begin{remark}
The condition $\|A-\tilde A_k\|<\sigma_k/2$ in Theorem \ref{thm:err-trsvd} requires a sufficiently accurate low-rank
RSVD approximation $(\tilde U_k,\tilde \Sigma_k,\tilde V_k)$ to $A$, i.e., the rank $k$ is sufficiently large. It enables one to
define a TRSVD solution $\tilde x_k$ of truncation level $k$.
\end{remark}

Next we give a relative error estimate for $\tilde x_k$ with respect to the TSVD approximation $x_k$. Such
an estimate was the focus of a few works \cite{BoutsidisMagdon:2014,XiangZou:2013,XiangZou:2015,WeiXieZhang:2016}.
First, we give a bound on $\|\tilde A_k\tilde A_k^*(A_k^*)^\dag-A_k\|$.

\begin{lemma}\label{lem:est-trsvd}
The following error estimate holds
\begin{equation*}
  \|\tilde A_k\tilde A_k^*(A_k^*)^\dag-A_k\| \leq \left(1+\sigma_1\sigma_k^{-1}\right)\|A_k-\tilde A_k\|.
\end{equation*}
\end{lemma}
\begin{proof}
This estimate follows by direct computation:
\begin{align*}
  \|\tilde A_k\tilde A_k^*(A_k^*)^\dag-A_k\|  & = \|[\tilde A_k\tilde A_k^* - A_kA_k^*](A_k^*)^\dag\|\\
    & \leq \|\tilde A_k(\tilde A_k^* - A_k^*)(A_k^*)^\dag\| + \|(\tilde A_k -A_k)A_k^*(A_k^*)^\dag\|\\
    & \leq \|\tilde A_k\|\|\tilde A_k^*-A_k^*\|\|(A_k^*)^\dag\| + \|\tilde A_k-A_k\|\|A_k^*(A_k^*)^\dag\|\\
    & \leq (\sigma_1\sigma_k^{-1}+1)\|\tilde A_k-A_k\|,
\end{align*}
since $\|\tilde A_k\| = \|\tilde P_kA\|\leq \|A\|=\sigma_1$.
Then the desired assertion follows directly.\qed
\end{proof}

Next we derive a relative error estimate between the approximations $x_k$ and $\tilde x_k$.
\begin{theorem}\label{thm:err-TSVD}
For any $k<r$, and $\|A-\tilde A_k\|<\sigma_{k}/2$, there holds
\begin{equation*}
  \frac{\|x_k-\tilde x_k\|}{\|x_k\|}\leq 4\Big(1+\frac{\sigma_1}{\sigma_k}\Big)\frac{\|A_k-\tilde A_k\|}{\sigma_k}.
\end{equation*}
\end{theorem}
\begin{proof}
We rewrite the TSVD solution $x_k$ as
\begin{equation}\label{eqn:x_k}
  x_k = A^* (A_kA_k^*)^{\dag}b = A_k^*(A_kA_k^*)^\dag b.
\end{equation}
By Lemma \ref{lem:Weyl} and the assumption $\|A-\tilde A_k\|<\sigma_{k}/2$, we have
$\tilde\sigma_{k} >0.$ Then $x_k-\tilde x_k = A^*((A_kA_k^*)^{\dag}-(\tilde A_k\tilde A_k^*)^{\dag})b$.
By Lemma \ref{lem:pseudoinverse2},
\begin{equation*}
   (A_kA_k^*)^\dag-(\tilde A_k\tilde A_k^*)^\dag = (\tilde A_k\tilde A_k^*)^\dag(\tilde A_k\tilde A_k^* - A_kA_k^*)(A_kA_k^*)^\dag.
\end{equation*}
It follows from the identity $(A_kA_k^*)^\dag=(A_k^*)^\dag A_k^\dag$ and \eqref{eqn:x_k} that
\begin{equation*}
  \begin{aligned}
    x_k-\tilde x_k &= A^*(\tilde A_k\tilde A_k^*)^\dag(\tilde A_k\tilde A_k^* - A_kA_k^*)(A_kA_k^*)^\dag b \\
      &= A^*(\tilde A_k\tilde A_k^*)^\dag(\tilde A_k\tilde A_k^*(A_k^*)^\dag - A_k)A_k^*(A_kA_k^*)^\dag b \\
      & = A^*(\tilde A_k^*)^\dag\tilde A_k^\dag(\tilde A_k\tilde A_k^*(A_k^*)^\dag - A_k)x_k.
  \end{aligned}
\end{equation*}
Thus, we obtain
\begin{equation*}
  \frac{\|x_k-\tilde x_k\|}{\|x_k\|}\leq \|A^*(\tilde A_k^*)^\dag\|\|\tilde A_k^\dag\|\|\tilde A_k\tilde A_k^* (A_k^*)^\dag-A_k\|.
\end{equation*}
By Lemma \ref{lem:Weyl}, we bound the term $\|\tilde A_k^\dag\|$ by $\|\tilde A_k^\dag\| \leq 2\sigma_k^{-1}.$
Combining the preceding estimates with Lemmas \ref{lem:est-AA_k} and \ref{lem:est-trsvd} completes the proof.\qed
\end{proof}

\begin{remark}
The relative error is determined by $k$ {\rm(}and
in turn by $\delta$ etc{\rm)}. Due to the presence of the factor $\sigma_k^{-2}$,
the estimate requires a highly accurate low-rank approximation, i.e., $\|A_k-\tilde A_k\|\ll \sigma_k(A)^{-2}$, and hence
it is more pessimistic than Theorem \ref{thm:err-trsvd}.
The estimate is comparable with the perturbation estimate
for the TSVD 
\begin{equation*}
  \frac{\|x_k-\bar x_k\|}{\|x_k\|}\leq \frac{\sigma_1\|A_k-\tilde A_k\|}{\sigma_k-\|A_k-\tilde A_k\|}\left(\frac{1}{\sigma_1}+\frac{\|Ax_k-b\|}{\sigma_k\|b\|}\right)+\frac{\|A_k-\tilde A_k\|}{\sigma_k}.
\end{equation*}
Modulo the $\alpha$ factor, the estimates in \cite{XiangZou:2013,WeiXieZhang:2016} for Tikhonov regularization
also depend on $\sigma_k^{-2}$ {\rm(}but can be much milder for a large $\alpha${\rm)}.
\end{remark}

\subsection{Tikhonov regularization}
The following bounds are useful for deriving error estimate on $\tilde x_\alpha$ in \eqref{eqn:t-tikh_x}.
\begin{lemma}\label{lem:tikh}
The following estimates hold
\begin{equation*}
  \begin{aligned}
    \|(AA^*+\alpha I)(\tilde A_k\tilde A_k^*+\alpha I)^{-1}-I\| &\leq  2\alpha^{-1}\|A\|\|A-\tilde A_k\|,\\
    \|[(AA^*+\alpha I)(\tilde A_k\tilde A_k^*+\alpha I)^{-1}-I]AA^*\|&\leq 2\|A\|(2\alpha^{-1}\|A\|\|A-\tilde A_k\|+1)\|A-\tilde A_k\|.
  \end{aligned}
\end{equation*}
\end{lemma}
\begin{proof}
It follows from the identity
\begin{equation*}
(AA^*+\alpha I)(\tilde A_k\tilde A_k^*+\alpha I)^{-1}-I
= (AA^*-\tilde A_k\tilde A_k^*)(\tilde A_k\tilde A_k^*+\alpha I)^{-1}
\end{equation*}
and the inequality $\|\tilde A_k\|=\|\tilde P_kA\|\leq \|A\|$ that
\begin{align*}
  \|(AA^*+\alpha I)(\tilde A_k\tilde A_k^*+\alpha I)^{-1}-I\|&\leq \alpha^{-1}(\|A\|+\|\tilde A_k\|)\|A-\tilde A_k\|\\
   &\leq 2\alpha^{-1}\|A\|\|A-\tilde A_k\|.
\end{align*}
Next, by the triangle inequality,
\begin{align*}
  &\|[(AA^*+\alpha I)(\tilde A_k\tilde A_k^*+\alpha I)^{-1}-I]AA^*\|\\
  \leq&\|AA^*-\tilde A_k\tilde A_k^*\|(\|(\tilde A_k\tilde A_k^*+\alpha I)^{-1}(AA^*+\alpha I)\|+\alpha\|(\tilde A_k\tilde A_k^*+\alpha I)^{-1}\|).
\end{align*}
This, together with the identity $AA^*-\tilde A_k\tilde A_k^*=A(A^*-\tilde A_k^*)+(A-\tilde A_k)A_k^*$
and the first estimate, yields the second estimate, completing the proof of the lemma.\qed
\end{proof}

Now we can give an error estimate on $\tilde x_\alpha$ in \eqref{eqn:t-tikh_x} under condition \eqref{eqn:range-ex}.
\begin{theorem}\label{thm:err-Tikh}
If condition \eqref{eqn:range-ex} holds, then the estimate $\tilde x_\alpha$ satisfies
\begin{equation*}
  \|\tilde x_\alpha - x^\dag\|  \leq \alpha^{-\frac32}\|A\|\|A-\tilde A_k\|\big(\delta+(2\alpha^{-1}\|A\|\|A-\tilde A_k\|+1)\alpha\|w\|\big) +2^{-1}\alpha^{\frac12} \|w\|.
\end{equation*}
\end{theorem}
\begin{proof}
First, with condition \eqref{eqn:range-ex}, $x^\dag$ can be rewritten as
\begin{equation*}
  \begin{aligned}
  x^\dag & = (A^*A+\alpha I)^{-1}(A^*A+\alpha I)x^\dag =(A^*A+\alpha I)^{-1}(A^* b^\dag + \alpha x^\dag)\\
   & = (A^*A+\alpha I)^{-1}A^*(b^\dag + \alpha w).
  \end{aligned}
\end{equation*}
The identity \eqref{eqn:matrix-id} implies $x^\dag = A^*(AA^*+\alpha I)^{-1}(b^\dag +\alpha w)$.
Consequently,
\begin{align*}
    &\quad\tilde x_\alpha - x^\dag = A^*[(\tilde A_k\tilde A_k^*+\alpha I)^{-1}b-(AA^*+\alpha I)^{-1}(b^\dag +\alpha w)]\\
      & = A^*[(\tilde A_k\tilde A_k^*+\alpha I)^{-1}e + ((\tilde A_k \tilde A_k^*+\alpha I)^{-1}-(AA^*+\alpha I)^{-1})b^\dag - \alpha (AA^*+\alpha I)^{-1}w].
\end{align*}
Let $\tilde I =(AA^*+\alpha I)(\tilde A_k\tilde A_k^*+\alpha I)^{-1}$. Then
by the identity \eqref{eqn:matrix-id}, there holds
\begin{equation*}
  \begin{aligned}
    (A^*A+\alpha I)(\tilde x_\alpha -x^\dag)=A^*[\tilde Ie+(\tilde I-I)b^\dag - \alpha w],
  \end{aligned}
\end{equation*}
and taking inner product with $\tilde x_\alpha-x^\dag$ yields
\begin{equation*}
  \|A(\tilde x_\alpha-x^\dag)\|^2 +\alpha\|\tilde x_\alpha -x^\dag\|^2 \leq \big(\|\tilde Ie\|+\|(\tilde I-I)b^\dag\| + \alpha\|w\|\big)\|A(\tilde x_\alpha -x^\dag)\|.
\end{equation*}
By Young's inequality $ab\leq \frac{1}{4}a^2+b^2$ for any $a,b\in\mathbb{R}$, we deduce
\begin{equation*}
  \alpha^{\frac12} \|\tilde x_\alpha -x^\dag\| \leq 2^{-1}(\|\tilde Ie\|+\|(\tilde I-I)b^\dag\| + \alpha\|w\|).
\end{equation*}
By Lemma \ref{lem:tikh} and the identity $ b^\dag=AA^*w$, we have
\begin{align*}
  \|\tilde Ie\| & \leq 2\alpha^{-1}\|A\|\|A-\tilde A_k\|\delta,\\
  \|(\tilde I-I)b^\dag\| &= \|(\tilde I-I)AA^*w\| \\
    &\leq 2\|A\|(2\alpha^{-1}\|A\|\|A-\tilde A_k\|+1)\|A-\tilde A_k\|\|w\|.
\end{align*}
Combining the preceding estimates yield the desired assertion.\qed
\end{proof}
\begin{remark}
To maintain the error $\|\tilde x_\alpha-x^\dag\|$, the accuracy
of $\tilde A_k$ should be of $O(\delta)$, and $\alpha$ should be of $O(\delta)$,
which gives an overall accuracy $O(\delta^{1/2})$. The tolerance on $\|A-\tilde A_k\|$
can be relaxed for high noise levels. It is consistent with existing theory for Tikhonov
regularization with noisy operators \cite{MaassRieder:1997,Neubauer:1988,Tautenhahn:2008}.
\end{remark}

\begin{remark}
The following relative error estimate was shown \cite[Theorem 1]{XiangZou:2013}:
\begin{equation*}
  \frac{\|x_\alpha-\hat x_\alpha\|}{\|x_\alpha\|}\leq c(2\sec\theta\kappa+\tan\theta\kappa^2)\sigma_{k+1}+O(\sigma_{k+1}^2),
\end{equation*}
with $\theta=\sin^{-1}\frac{(\|b-Ax_\alpha\|^2+\alpha\|x_\alpha\|^2)^{\frac{1}{2}}}{\|b\|}$ and $\kappa=
(\sigma_1^2+\alpha)(\frac{\alpha^\frac12}{\sigma_n^2+\alpha}+\max_{1\leq i\leq n}\frac{\sigma_i}{\sigma_i^2
+\alpha})$. $\kappa$ is a variant of condition number. Thus, $\tilde A_k$ should approximate accurately
$A$ in order not to spoil the accuracy,
and the estimate can be pessimistic for small $\alpha$ for which the estimate tends to blow up.
\end{remark}
\subsection{General Tikhonov regularization}
Last, we give an error estimate for $\tilde x_\alpha$ defined in \eqref{eqn:t-tikh_x-gen} under the following condition
\begin{equation}\label{eqn:range-ex-gen}
  x^\dag = \Gamma A^*w,
\end{equation}
where $\mathcal{N}(L)=\{0\}$, and $\Gamma=L^\dag L^{*\dag}$. Also recall that $B=A\Gamma^\dag$.
\begin{theorem}\label{thm:err-gtikh}
If Condition \eqref{eqn:range-ex-gen} holds, then the regularized solution $\tilde x_\alpha$ in \eqref{eqn:t-tikh_x-gen} satisfies
\begin{equation*}
   \|L(x^\dag -\tilde x_\alpha)\|  \leq \alpha^{-\frac32}\|B\|\|B-\tilde B_k\|\big(\delta+(2\alpha^{-1}\|B\|\|B-\tilde B_k\|+1)\alpha\|w\|\big) + 2^{-1}\alpha^\frac{1}{2} \|w\|.
\end{equation*}
\end{theorem}
\begin{proof}
First, by the source condition \eqref{eqn:range-ex-gen}, we rewrite $x^\dag$ as
\begin{equation*}
  \begin{aligned}
  x^\dag & = (A^*A+\alpha L^*L)^{-1}(A^*A+\alpha L^*L)x^\dag\\
   & = (A^*A+\alpha L^*L)^{-1}(A^*b^\dag + \alpha A^*w).
  \end{aligned}
\end{equation*}
Now with the identity $(A^*A+\alpha L^*L)^{-1}A^* = \Gamma A^*(A\Gamma A^*+\alpha I)^{-1},$ we have
\begin{equation*}
  x^\dag = \Gamma A^*(A\Gamma A^*+\alpha I)^{-1}(b^\dag +\alpha w).
\end{equation*}
Thus, upon recalling $B=AL^\dag$, we have
\begin{equation*}
  \begin{aligned}
    &\quad \tilde x_\alpha - x^\dag  = \Gamma A^*[(\tilde B_k\tilde B_k^*+\alpha I)^{-1}b-(BB^*+\alpha I)^{-1}(b^\dag +\alpha w)]\\
      & = \Gamma A^*[(\tilde B_k\tilde B_k^*+\alpha I)^{-1}e + ((\tilde B_k\tilde B_k^*+\alpha I)^{-1}-(BB^*+\alpha I)^{-1})b^\dag - \alpha (BB^*+\alpha I)^{-1}w].
  \end{aligned}
\end{equation*}
It follows from the identity
\begin{equation*}
(A^*A+\alpha L^*L)\Gamma A^* = (A^*A+\alpha L^*L)L^{\dag}L^{\dag*}A^* = A^*(BB^*+\alpha I),
\end{equation*}
that
\begin{equation*}
    (A^*A+\alpha L^*L)(\tilde x_\alpha -x^\dag)= A^*[\tilde Ie+(\tilde I-I)b^\dag - \alpha w],
\end{equation*}
with $\tilde I =(BB^*+\alpha I)(\tilde B_k\tilde B_k^*+\alpha I)^{-1}$.
Taking inner product with $x_\alpha-x^\dag$ and applying Cauchy-Schwarz inequality yield
\begin{equation*}
  \|A(\tilde x_\alpha-x^\dag)\|^2 +\alpha\|L(\tilde x_\alpha -x^\dag)\|^2 \leq (\|\tilde Ie\|+\|(\tilde I-I)b^\dag\| + \|\alpha w\|)\|A(\tilde x_\alpha -x^\dag)\|,
\end{equation*}
Young's inequality implies
$\alpha^\frac{1}{2} \|L(\tilde x_\alpha -x^\dag)\| \leq 2^{-1}(\|\tilde Ie\|+\|(\tilde I-I)b^\dag\| + \alpha \|w\|)$.
The identity $b^\dag = Ax^\dag = AL^{\dag}L^{\dag*}A^*w=BB^*w$ from
\eqref{eqn:range-ex-gen} and Lemma \ref{lem:tikh} complete the proof.\qed
\end{proof}

\section{Numerical experiments and discussions}\label{sec:numer}

Now we present numerical experiments to illustrate our approach. The noisy
data $b$ is generated from the exact data $b^\dag$ as follows
\begin{equation*}
  b_i = b_i^\dag + \delta \max_{j}(|b_j^\dag|)\xi_i,\quad i =1,\ldots,n,
\end{equation*}
where $\delta$ is the relative noise level, and the random variables $\xi_i$s follow the standard Gaussian
distribution. All the computations were carried out on a personal laptop with 2.50 GHz CPU and 8.00G RAM
by \texttt{MATLAB} 2015b. When implementing Algorithm \ref{alg:rsvd}, the default choices $p=5$ and $q=0$
are adopted. Since the TSVD and Tikhonov solutions are close for suitably chosen regularization
parameters, we present only results for Tikhonov regularization (and the general case
with $L$ given by the first-order difference, which has a one-dimensional
kernel $\mathcal{N}(L)$). 

Throughout, the regularization parameter $\alpha$ is determined by uniformly sampling an interval on a logarithmic
scale, and then taking the value attaining the smallest reconstruction error, where approximate Tikhonov minimizers are found
by either \eqref{eqn:t-tikh_x} or \eqref{eqn:t-tikh_x-gen2} with a large $k$ ($k=100$ in all the experiments).

\subsection{One-dimensional benchmark inverse problems}
First, we illustrate the efficiency and accuracy of proposed approach, and compare it with existing
approaches \cite{XiangZou:2013,XiangZou:2015}. We consider seven examples (i.e., \texttt{deriv2}, \texttt{heat},
\texttt{phillips}, \texttt{baart}, \texttt{foxgood}, \texttt{gravity} and \texttt{shaw}),
taken from the popular public-domain \texttt{MATLAB} package \textbf{regutools} (available from
\url{http://www.imm.dtu.dk/~pcha/Regutools/}, last accessed on January 8, 2019), which have been used
in existing studies (see, e.g., \cite{WeiXieZhang:2016,XiangZou:2013,XiangZou:2015}). They are Fredholm integral
equations of the first kind, with the first three examples being mildly ill-posed (i.e., $\sigma_i$s decay algebraically)
and the rest severely ill-posed (i.e., $\sigma_i$s decay exponentially). Unless otherwise stated, the
examples are discretized with a dimension $n=m=5000$. The resulting matrices are dense and unstructured.
The rank $k$ of $\tilde A_k$ is fixed at $k=20$, which is sufficient to for all examples.

The numerical results by standard Tikhonov regularization and two randomized variants, i.e.,
\eqref{eqn:t-tikh_x-xz} and \eqref{eqn:t-tikh_x}, for the examples
are presented in Table \ref{tab:std}. The accuracy of the approximations, i.e., the
Tikhonov solution $x_\alpha$, and two randomized approximations $\hat x_\alpha$ (cf. \eqref{eqn:t-tikh_x-xz},
proposed in \cite{XiangZou:2013}) and $\tilde x_\alpha$ (cf. \eqref{eqn:t-tikh_x}, the proposed in this work), is measured in two different ways:
\begin{align*}
  \tilde e_{xz} & = \|\hat x_\alpha - x_\alpha\|,\quad \tilde e_{ij} = \|\tilde x_\alpha- x_\alpha\|,\\
   e & = \| x_\alpha-x^\dag\|,\quad e_{xz} = \|\hat x_\alpha- x^\dag\|, \quad e_{ij} = \|\tilde x_\alpha-x^\dag\|,
\end{align*}
where the methods are indicated by the subscripts.
That is, $\tilde e_{xz}$ and $\tilde e_{ij}$ measure the accuracy with respect to the Tikhonov solution
$x_\alpha$, and $e$, $e_{xz}$
and $e_{ij}$ measure the accuracy with respect to the exact one $x^\dag$.

The following observations can be drawn from Table \ref{tab:std}. For all examples, the three approximations
$x_\alpha$, $\tilde x_\alpha$ and $\hat x_\alpha$ have comparable accuracy relative to the exact solution $x^\dag$,
and the errors $e_{ij}$ and $e_{xz}$ are fairly close to the error $e$ of the Tikhonov solution $x_\alpha$. Thus,
RSVD can maintain the reconstruction accuracy. For
\texttt{heat}, despite the apparent large magnitude of the errors $\tilde e_{xz}$ and $\tilde e_{ij}$, the errors
$e_{xz}$ and $e_{ij}$ are not much worse than $e$. A close inspection shows that the difference of the
reconstructions are mostly in the tail part, which requires more
modes for a full resolution. The computing time (in seconds)
for obtaining $x_\alpha$ and $\tilde x_\alpha$ and $\hat x_\alpha$ is about $6.60$, $0.220$ and $0.220$, where
for the latter two, it includes also the time for computing RSVD. Thus, for all the examples, with a rank $k=20$,
RSVD can accelerate standard Tikhonov regularization by a factor of 30, while maintaining the
accuracy, and the proposed approach is competitive with the
one in \cite{XiangZou:2013}. Note that the choice $k=20$ can be
greatly reduced for severely ill-posed problems; see Section \ref{ssec:conv} below for discussions.

\begin{table}[hbt!]
  \centering
  \caption{Numerical results by standard Tikhonov regularization at two noise levels.\label{tab:std}}
  \begin{tabular}{lllllll}
  \hline
  example &    $\delta$   &   $\tilde e_{xz}$ & $\tilde e_{ij}$ & $e$ & $e_{xz}$ & $e_{ij}$ \\
  \hline
  \texttt{baart}   & $1\%$ & 1.14e-9 &  1.14e-9  & 1.68e-1 &  1.68e-1 &  1.68e-1\\
                   & $5\%$ & 5.51e-11&  6.32e-11 & 2.11e-1 &  2.11e-1 &  2.11e-1\\
  \hline
  \texttt{deriv2}  & $1\%$ & 2.19e-2 &  2.41e-2  & 1.18e-1 &  1.20e-1 &  1.13e-1\\
                   & $5\%$ & 1.88e-2 &  2.38e-2  & 1.59e-1 &  1.60e-1 &  1.62e-1\\
  \hline
  \texttt{foxgood} & $1\%$ & 2.78e-7 &  2.79e-7  & 4.93e-1 &  4.93e-1 &  4.93e-1\\
                   & $5\%$ & 1.91e-7 &  1.96e-7  & 1.18e0  &  1.18e0  &  1.18e0\\
  \hline
  \texttt{gravity} & $1\%$ & 1.38e-4 &  1.41e-4  & 7.86e-1 &  7.86e-1 &  7.86e-1\\
                   & $5\%$ & 1.83e-4 &  1.84e-4  & 2.63e0  &  2.63e0  &  2.63e0\\
  \hline
  \texttt{heat}    & $1\%$ & 1.33e0  &  1.13e0   & 9.56e-1 &  1.67e0  &  1.50e0\\
                   & $5\%$ & 9.41e-1 &  9.45e-1  & 2.02e0  &  1.70e0  &  1.99e0\\
  \hline
  \texttt{phillips}& $1\%$ & 5.53e-3 &  4.09e-3  & 6.28e-2 &  6.19e-2 &  6.24e-2\\
                   & $5\%$ & 6.89e-3 &  7.53e-3  & 9.57e-2 &  9.53e-2 &  9.79e-2\\
  \hline
  \texttt{shaw}    & $1\%$ & 3.51e-9 &  3.49e-9  & 4.36e0  &  4.36e0  &  4.36e0\\
                   & $5\%$ & 1.34e-9 &  1.37e-9  & 8.23e0  &  8.23e0  &  8.23e0\\
  \hline
  \end{tabular}
\end{table}

The preceding observations remain largely valid for general Tikhonov regularization; see
Table \ref{tab:h1}. Since the construction of the approximation $\hat x_\alpha$ does not retain the
structure of the regularized solution $x_\alpha$, the error $\tilde e_{xz}$ can potentially be much larger
than $\tilde e_{ij}$, which can indeed be observed. The errors $e$, $e_{xz}$ and $e_{ij}$
are mostly comparable, except for \texttt{deriv2}. For \texttt{deriv2}, the approximation
$\hat x_\alpha$ suffers from grave errors, since the projection of $L$ into $\mathcal{R}
(Q)$ is very inaccurate for preserving $L$. It is expected that the loss occurs whenever general Tikhonov
penalty is much more effective than the standard one. This shows the importance of structure
preservation. Note that, for a general $L$, $\tilde x_\alpha$ takes only about 1.5 times the
computing time of $\hat x_\alpha$. This
cost can be further reduced since $L$ is highly structured and admits fast inversion. Thus
preserving the range structure of $x_\alpha$ in \eqref{eqn:range} does not
incur much overhead.

\begin{table}[hbt!]
  \centering
  \caption{Numerical results by general Tikhonov regularization (with the first-order derivative penalty) for the examples
  at two noise levels.\label{tab:h1}}
  \begin{tabular}{lllllll}
  \hline
  example &         $\delta$ & $\tilde e_{xz}$ & $\tilde e_{ij}$ & $e$ & $e_{xz}$ & $e_{ij}$ \\
  \hline
  \texttt{baart}   & $1\%$ & 3.35e-10 &  2.87e-10 &  1.43e-1 &  1.43e-1 &  1.43e-1\\
                   & $5\%$ & 3.11e-10 &  8.24e-12 &  1.48e-1 &  1.48e-1 &  1.48e-1\\
  \hline
  \texttt{deriv2}  & $1\%$ & 1.36e-1  &  4.51e-4  &  1.79e-2 &  1.48e-1 &  1.78e-2\\
                   & $5\%$ & 1.57e-1  &  3.85e-4  &  2.40e-2 &  1.77e-1 &  2.40e-2\\
  \hline
  \texttt{foxgood} & $1\%$ & 4.84e-2  &  2.26e-8  &  9.98e-1 &  1.02e0 &  9.98e-1\\
                   & $5\%$ & 1.90e-2  &  1.51e-9  &  2.27e0 &  2.28e0 &  2.27e0\\
  \hline
  \texttt{gravity} & $1\%$ & 3.92e-2  &  2.33e-5  &  1.39e0 &  1.41e0 &  1.39e0\\
                   & $5\%$ & 1.96e-2  &  9.47e-6  &  3.10e0 &  3.10e0 &  3.10e0\\
  \hline
  \texttt{heat}    & $1\%$ & 5.54e-1  &  8.74e-1  &  8.95e-1 &  1.06e0 &  1.32e0\\
                   & $5\%$ & 8.90e-1  &  1.01e0  &  1.87e0 &  1.76e0 &  1.99e0\\
  \hline
  \texttt{phillips}& $1\%$ & 3.25e-3  &  3.98e-4  &  6.14e-2 &  6.06e-2 &  6.14e-2\\
                   & $5\%$ & 5.64e-3  &  5.82e-4  &  8.37e-2 &  8.18e-2 &  8.34e-2\\
  \hline
  \texttt{shaw}    & $1\%$ & 3.79e-4  &  3.70e-8  &  3.32e0 &  3.32e0 &  3.32e0\\
                   & $5\%$ & 9.73e-4  &  2.17e-8  &  9.23e0 &  9.23e0 &  9.23e0\\
  \hline
  \end{tabular}
\end{table}

Last, we present some results on the computing time for \texttt{deriv2} versus the
problem dimension, and at two truncation levels for RSVD, i.e., $k=20$ and $k=30$. The numerical results are given in Fig.
\ref{fig:cpu}. The cubic scaling of the standard approach and quadratic scaling
of the approach based on RSVD are clearly observed,  confirming the complexity
analysis in Sections \ref{sec:prelim} and \ref{sec:reg}. In both \eqref{eqn:t-tikh_x} and
\eqref{eqn:t-tikh_x-gen2}, computing RSVD represents the dominant part of the overall computational
efforts, and thus the increase of the rank $k$ from 20 to 30 adds very little overheads
(compare the dashed and solid curves in Fig. \ref{fig:cpu}). Further, for Tikhonov regularization, the
two randomized variants are equally efficient, and for the general one, the
proposed approach is slightly more expensive due to its direct use of
$L$ in constructing the approximation $\tilde B_k$ to $B:=AL^\#$. Although
not presented, we note that the results for other examples are very similar.

\begin{figure}[hbt!]
  \centering
  \begin{tabular}{cc}
    \includegraphics[width=0.5\textwidth]{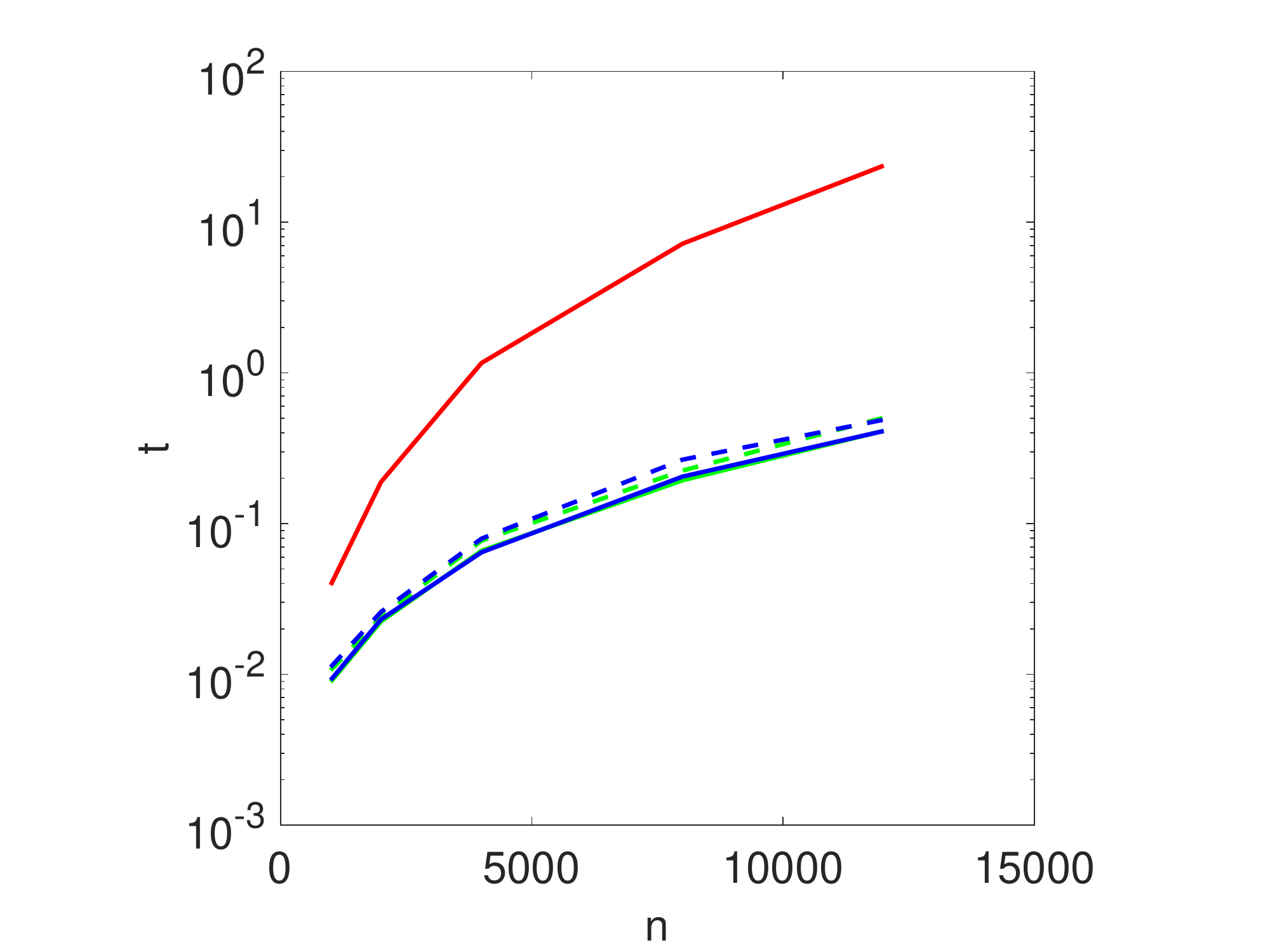} & \includegraphics[width=0.5\textwidth]{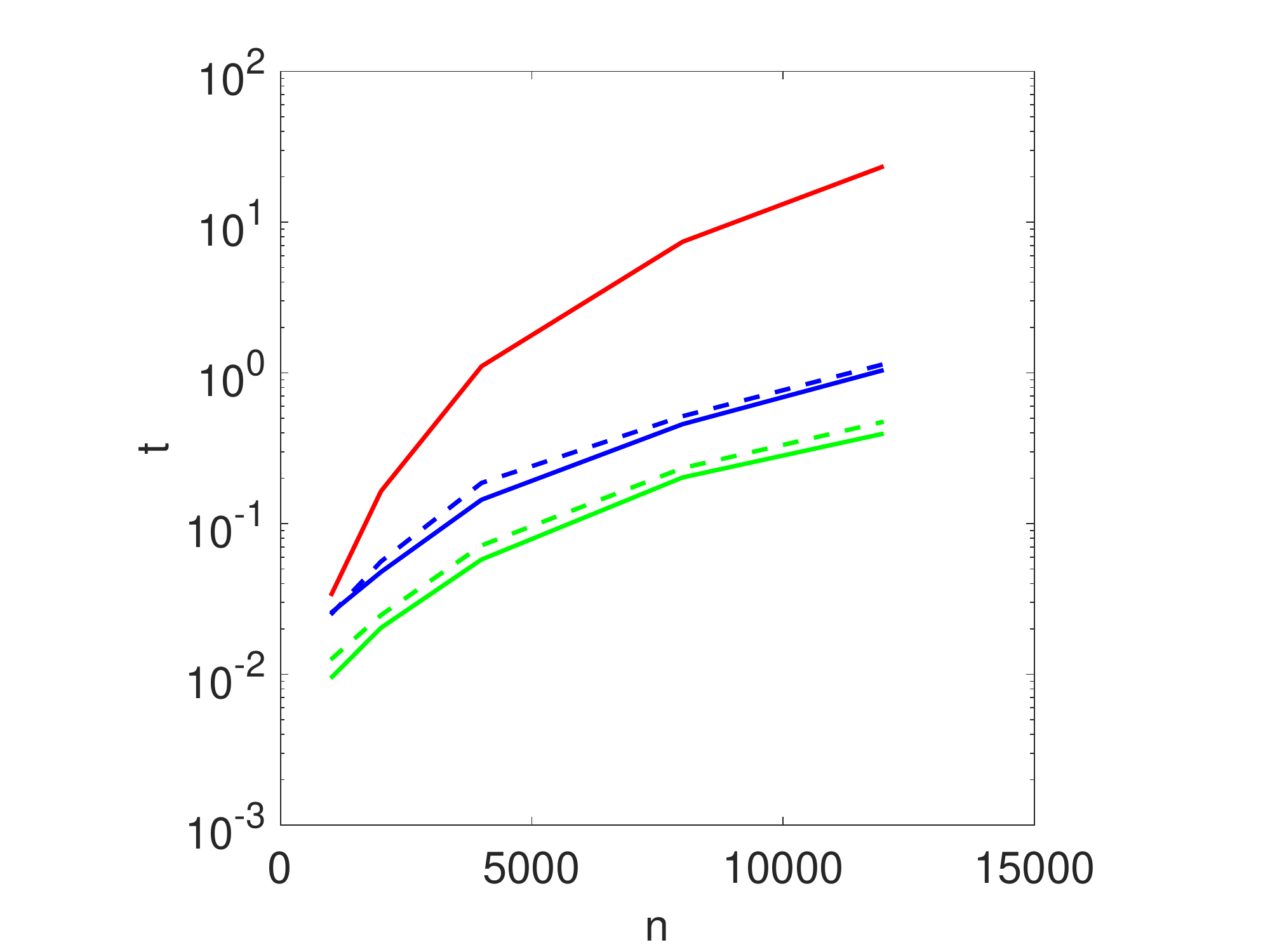}\\
    (a) standard Tikhonov & (b) general Tikhonov
  \end{tabular}
  \caption{The computing time $t$ (in seconds) for the example \texttt{deriv2} at different dimension $n$ ($m=n$).
  The red, green and blue curves refer to Tikhonov regularization, existing approach \cite{XiangZou:2013,XiangZou:2015}
  and the new approach, respectively, and the sold and dashed curves denote $k=20$ and $k=30$, respectively.\label{fig:cpu}}
\end{figure}

\subsection{Convergence of the algorithm}\label{ssec:conv}
There are several factors influencing the quality of $\tilde x_\alpha$
the regularization parameter $\alpha$, the noise level $\delta$ and the rank $k$ of the RSVD approximation.
The optimal truncation level $k$ should depend on both $\alpha$ and $\delta$. This part
presents a study with \texttt{deriv2} and \texttt{shaw}, which
are mildly and severely ill-posed, respectively.

\begin{figure}[hbt!]
  \centering
  \begin{tabular}{cc}
  \includegraphics[width=0.5\textwidth]{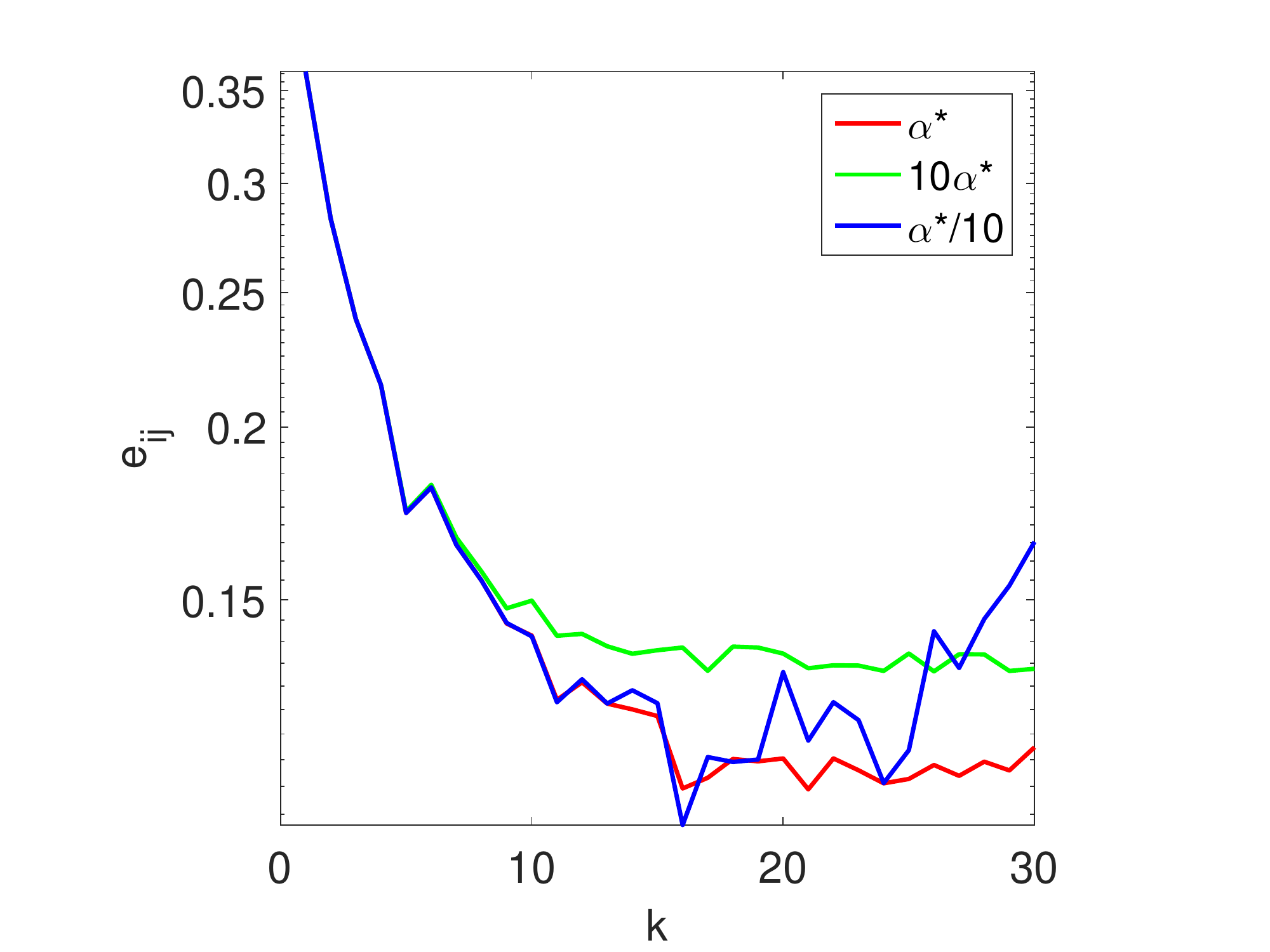} & \includegraphics[width=0.5\textwidth]{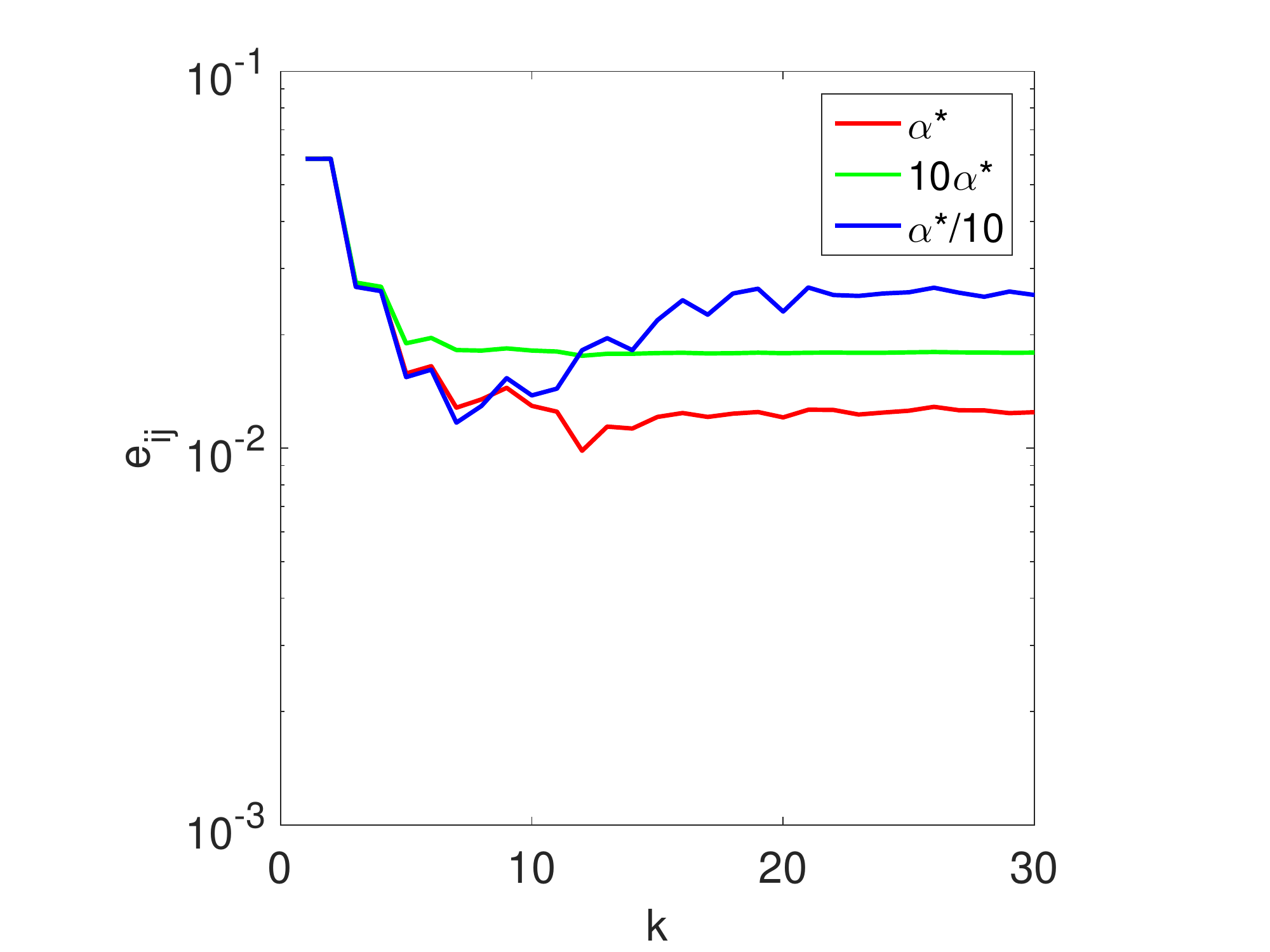}\\
  \includegraphics[width=0.5\textwidth]{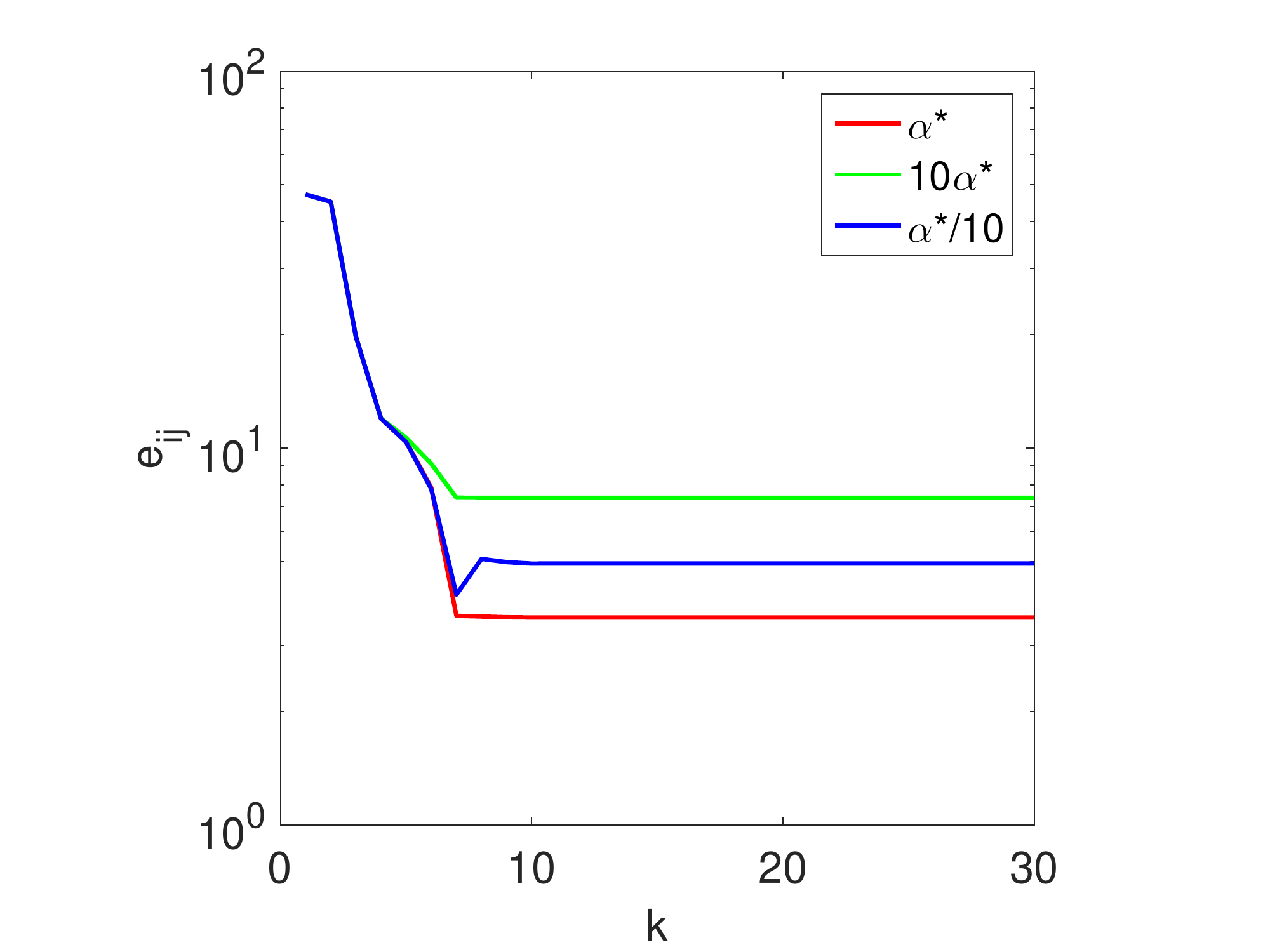} & \includegraphics[width=0.5\textwidth]{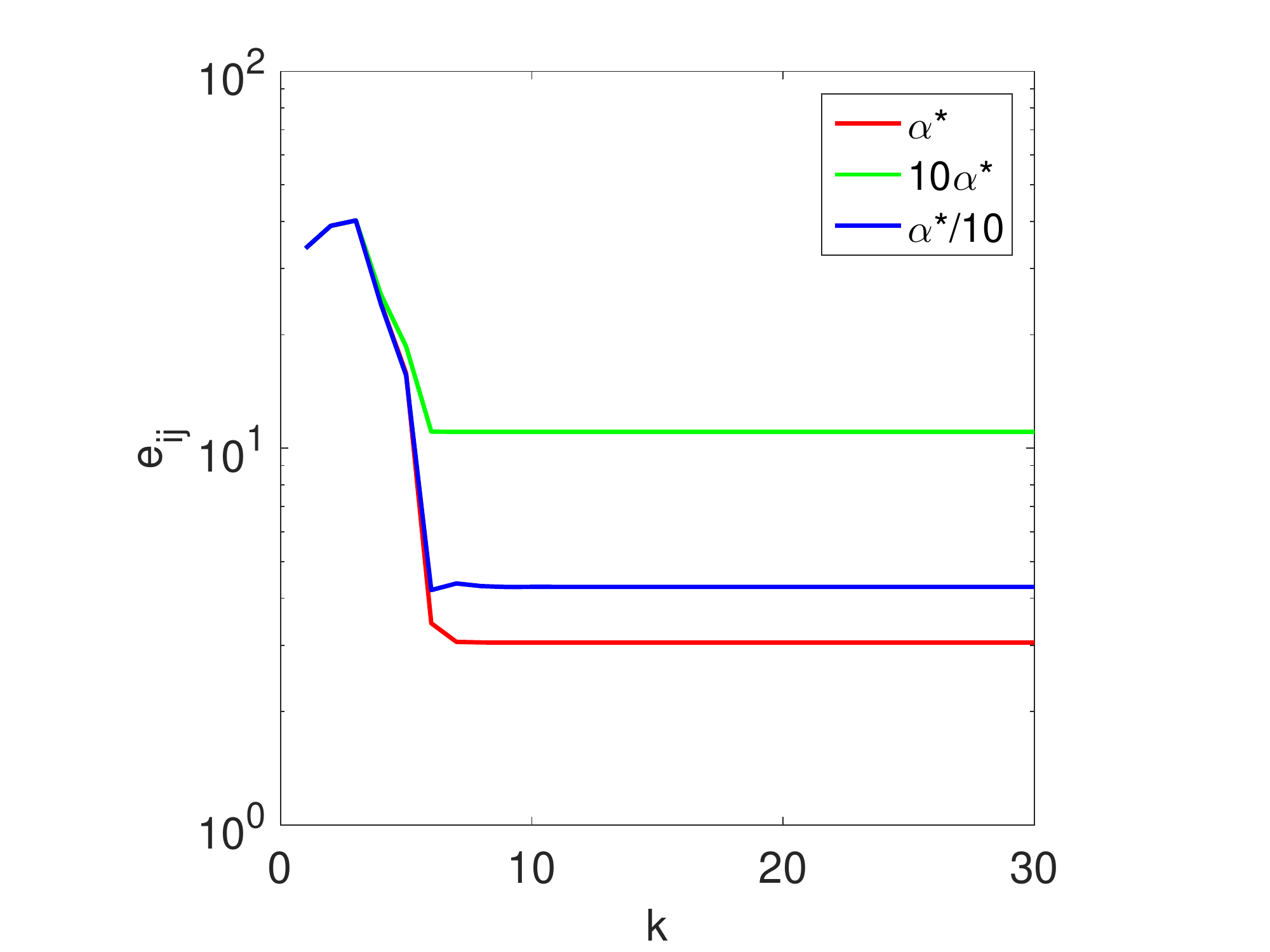}\\
  (a) standard Tikhonov & (b) general Tikhonov
  \end{tabular}
  \caption{The convergence of the error $e_{ij}$ with respect to the rank $k$ for \texttt{deriv2} (top) and \texttt{shaw} (bottom) with $\delta=1\%$ and different regularization parameters.\label{fig:convk}}
\end{figure}

First, we examine the influence of $\alpha$ on the optimal $k$. The numerical results for three different
levels of regularization are given in Fig. \ref{fig:convk}. In the figure, the notation $\alpha^*$ refers
to the value attaining the the smallest error for Tikhonov solution $x_\alpha$, and thus $10\alpha^*$ and
$\alpha^*/10$ represent respectively over- and under-regularization. The optimal
$k$ value decreases with the increase of $\alpha$ when $\alpha \gg\alpha^*$. This may be explained by the
fact that a too large $\alpha$ causes large approximation error and thus can tolerate large errors
in the approximation $\tilde A_k$ (for a small $k$). The dependence can be sensitive
for mildly ill-posed problems, and also on the penalty. The penalty influences the singular value
spectra in the RSVD approximation implicitly by preconditioning: since $L$ is a discrete differential operator,
the (weighted) pseudoinverse $L^\#$ (or $L^\dag$) is a smoothing operator, and thus the singular values of
$B=AL^\#$ decay faster than that of $A$. In all cases, the error $e_{ij}$ is
nearly monotonically decreasing in $k$ (and finally levels off at $e$, as expected). In the under-regularized
regime (i.e., $\alpha\ll\alpha^*$), the behavior is slightly different: the error $e_{ij}$ first decreases, and then increases before eventually
leveling off at $e$. This is attributed to the fact that proper low-rank truncation of $A$ induces extra
regularization, in a manner similar to TSVD in Section \ref{ssec:tsvd}. Thus,
an approximation that is only close to $x_\alpha$ (see e.g., \cite{BoutsidisMagdon:2014,XiangZou:2013,
XiangZou:2015,WeiXieZhang:2016}) is not necessarily close to $x^\dag$, when $\alpha$ is not chosen properly.

\begin{figure}[hbt!]
  \centering
  \begin{tabular}{cc}
  \includegraphics[width=0.5\textwidth]{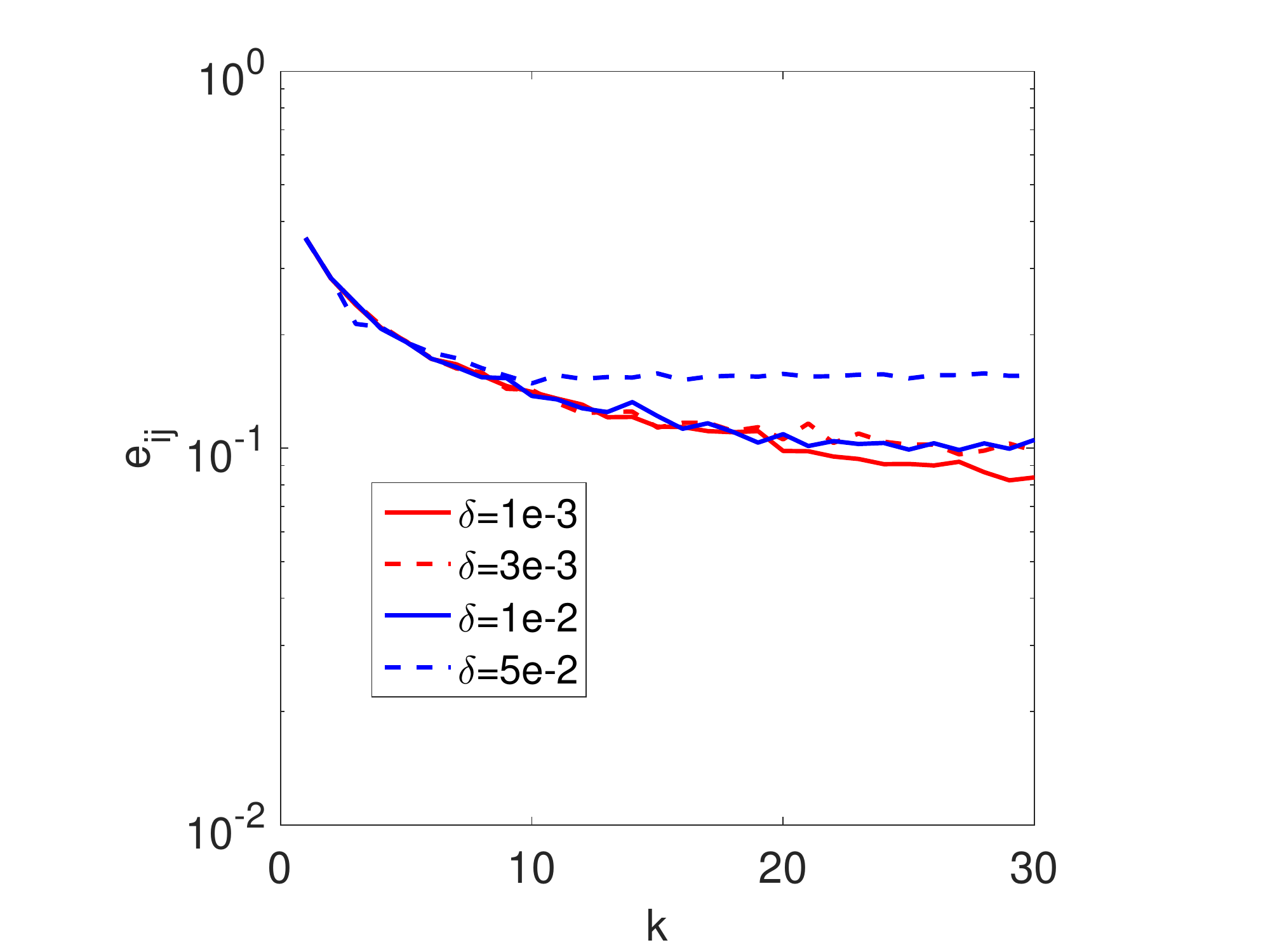} & \includegraphics[width=0.5\textwidth]{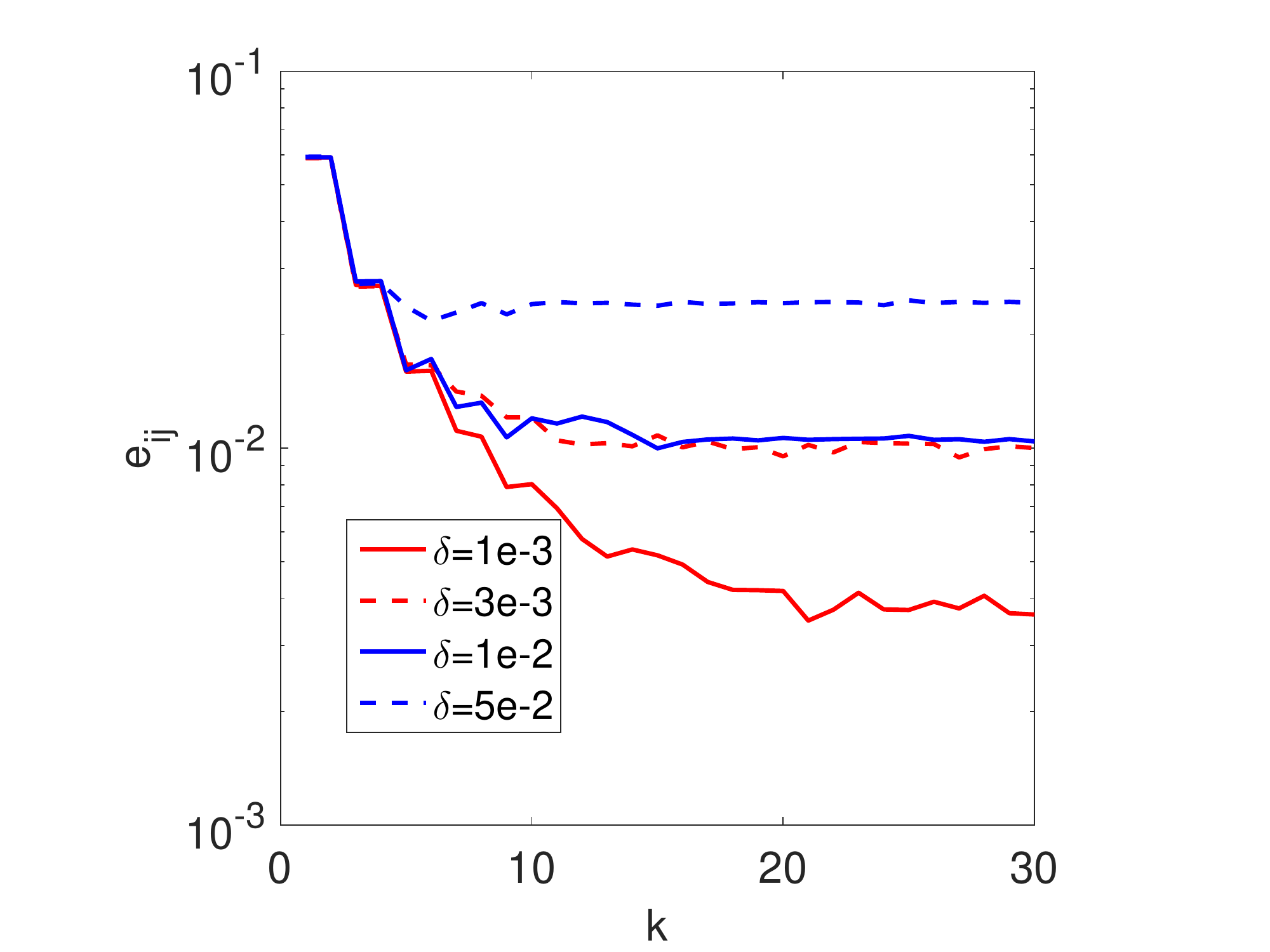}\\
  \includegraphics[width=0.5\textwidth]{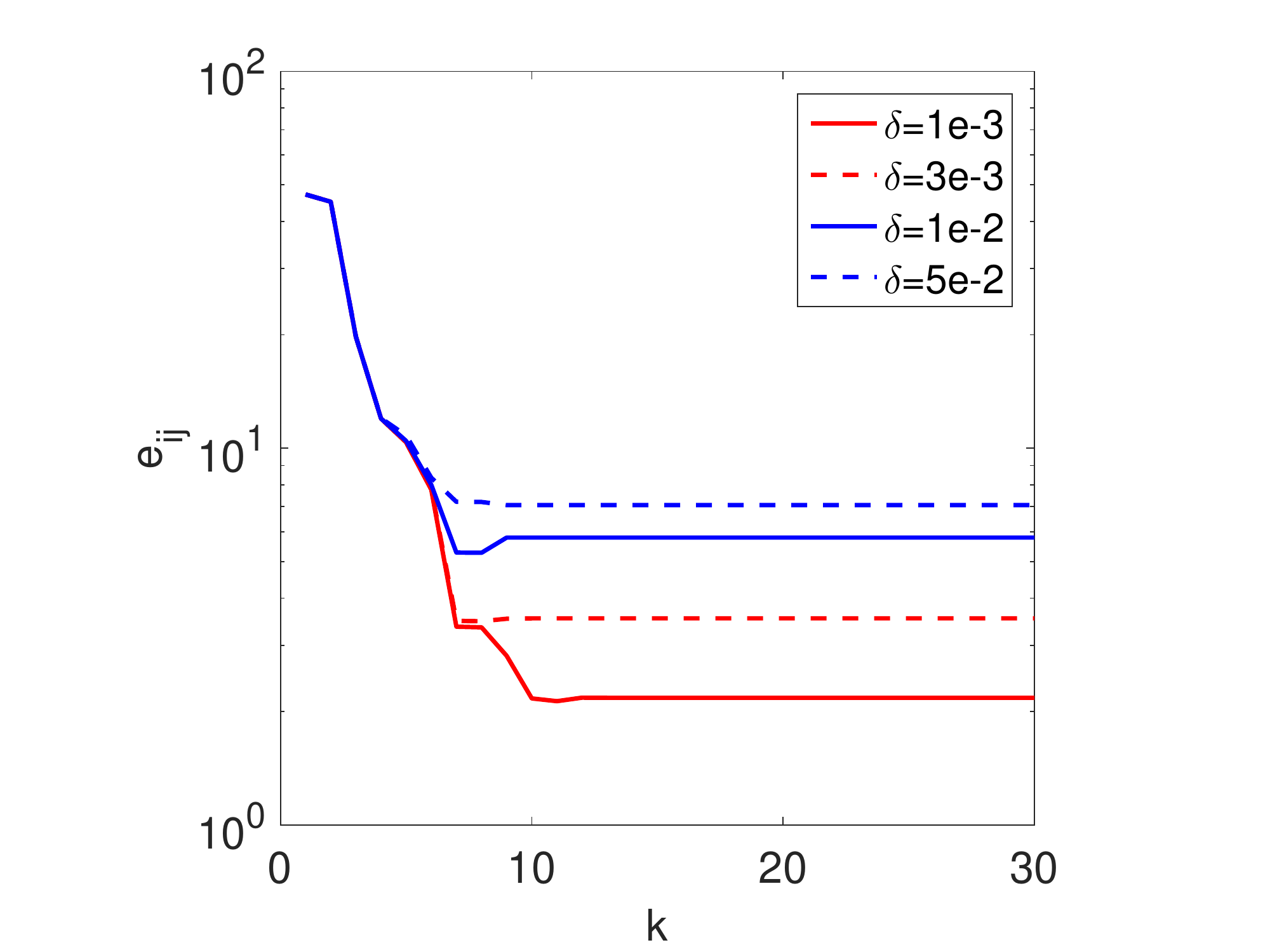} & \includegraphics[width=0.5\textwidth]{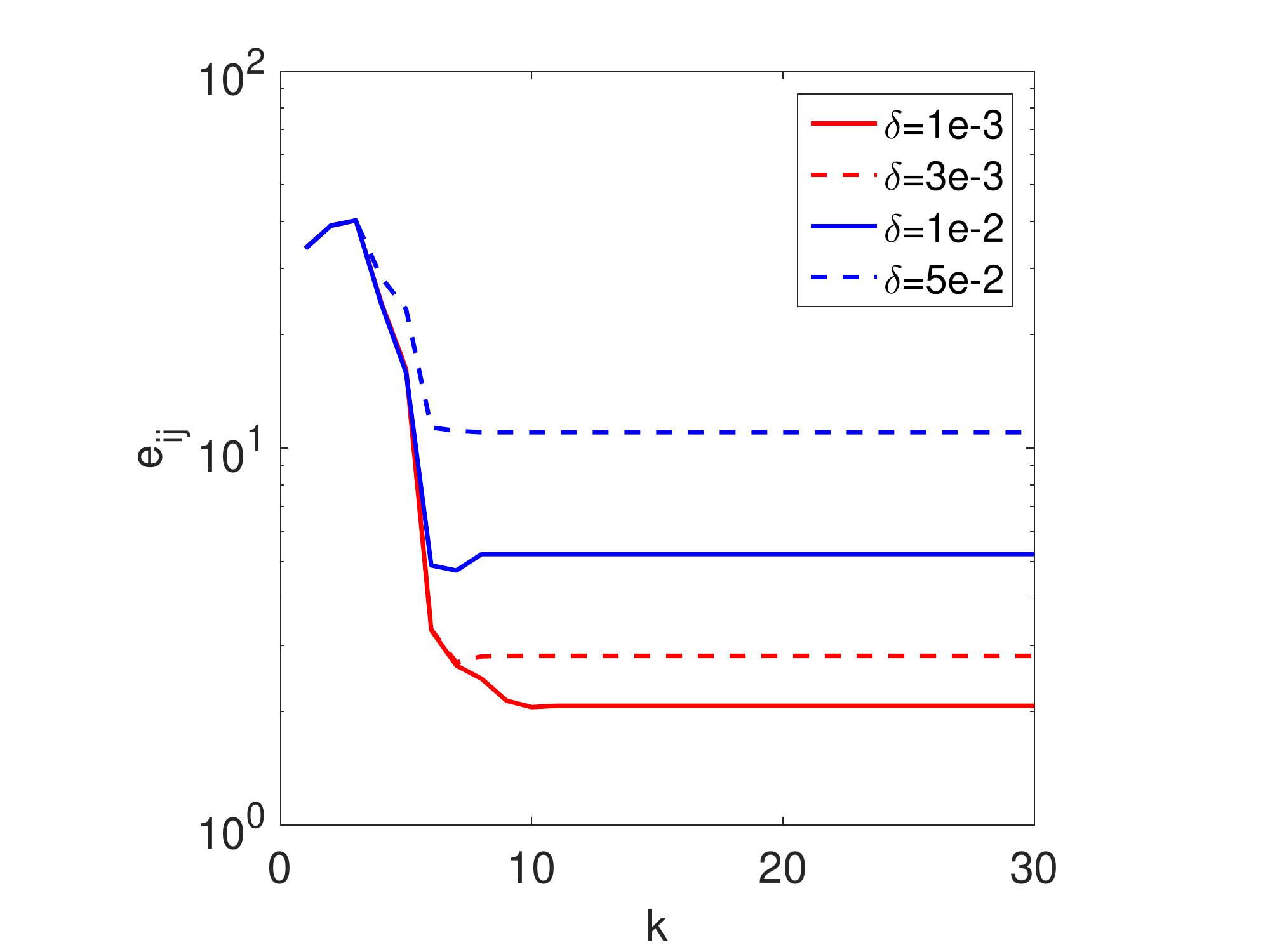}\\
  (a) standard Tikhonov & (b) general Tikhonov
  \end{tabular}
  \caption{The convergence of the error $e_{ij}$ with respect to the rank $k$ for \texttt{deriv2}
  (top) and \texttt{shaw} (bottom) at different noise levels. \label{fig:convk-nl}}
\end{figure}

Next we examine the influence of the noise level $\delta$; see Fig. \ref{fig:convk-nl}. With the optimal choice of $\alpha$, the
optimal $k$ increases as $\delta$ decreases, which is especially pronounced for mildly
ill-posed problems. Thus, RSVD is especially efficient for the following two cases: (a) highly
noisy data (b) severely ill-posed problem. These observations agree well with
Theorem \ref{thm:err-Tikh}: a low-rank approximation $\tilde A_k$ whose accuracy is commensurate with
$\delta$ is sufficient, and in either
case, a small rank is sufficient for obtaining an acceptable approximation. For a fixed
$k$, the error $e_{ij}$ almost increases monotonically with the noise level $\delta$.

These empirical observations naturally motivate developing an adaptive strategy for choosing the
rank $k$ on the fly so as to effect the optimal complexity. This requires a careful analysis of the
balance between $k$, $\delta$, $\alpha$, and suitable \textit{a posteriori} estimators. We leave this interesting topic to a future work.

\subsection{Electrical impedance tomography}\label{ssec:eit}
Last,  we illustrate the approach on 2D electrical impedance tomography (EIT), a diffusive
imaging modality of recovering the electrical conductivity from boundary voltage measurement. This is one
canonical nonlinear inverse problem. We consider the problem on a unit circle with sixteen electrodes
uniformly placed on the boundary, and adopt the complete electrode model \cite{SomersaloCheney:1992} as
the forward model. It is discretized by the standard Galerkin FEM with conforming piecewise linear basis
functions, on a quasi-uniform finite element mesh with $2129$ nodes. For the inversion step, we employ ten
sinusoidal input currents, unit contact impedance and measure the voltage data (corrupted
by $\delta=0.1\%$ noise). The reconstructions are obtained with an $H^1(\Omega)$-seminorm penalty.
We refer to \cite{GehreJinLu:2014,JinXuZou:2017} for details on numerical implementation.
We test the RSVD algorithm with the linearized model.
It can be implemented efficiently without explicitly computing the linearized map.
More precisely, let $F$ be the (nonlinear) forward operator, and $\sigma_0$ be the background (fixed at $1$). Then the random
probing of the range $\mathcal{R}(F'(\sigma_0))$ of the linearized forward operator $F'(\sigma_0)$ (cf. Step 4
of Algorithm \ref{alg:rsvd}) can be approximated by
\begin{equation*}
  F'(\sigma_0)\omega_i \approx F(\sigma_0+\omega_i) - F(\sigma_0), \quad i=1,\ldots k+p,
\end{equation*}
and it can be made very accurate by choosing a small variance for the random vector $\omega_i$.
Step 6 of Algorithm \ref{alg:rsvd} can be done efficiently via the adjoint technique.

The numerical results are presented in Fig. \ref{fig:eit}, where linearization refers to the reconstruction
by linearizing the nonlinear forward model at the background $\sigma_0$. This is one of the most classical reconstruction methods in
EIT imaging. The rank $k$ is taken to be $k=30$ for $\tilde x_\alpha$, which is
sufficient given the severe ill-posed nature of the EIT inverse problem. Visually, the RSVD reconstruction
is indistinguishable from the conventional approach. Note that contrast loss is often observed for
EIT reconstructions obtained by a smoothness penalty. The computing time (in seconds) for RSVD is less
than 8, whereas that for the conventional method is about 60. Hence,  RSVD can greatly accelerate
EIT imaging.

\begin{figure}
  \centering
  \begin{tabular}{ccc}
   \includegraphics[width=0.33\textwidth]{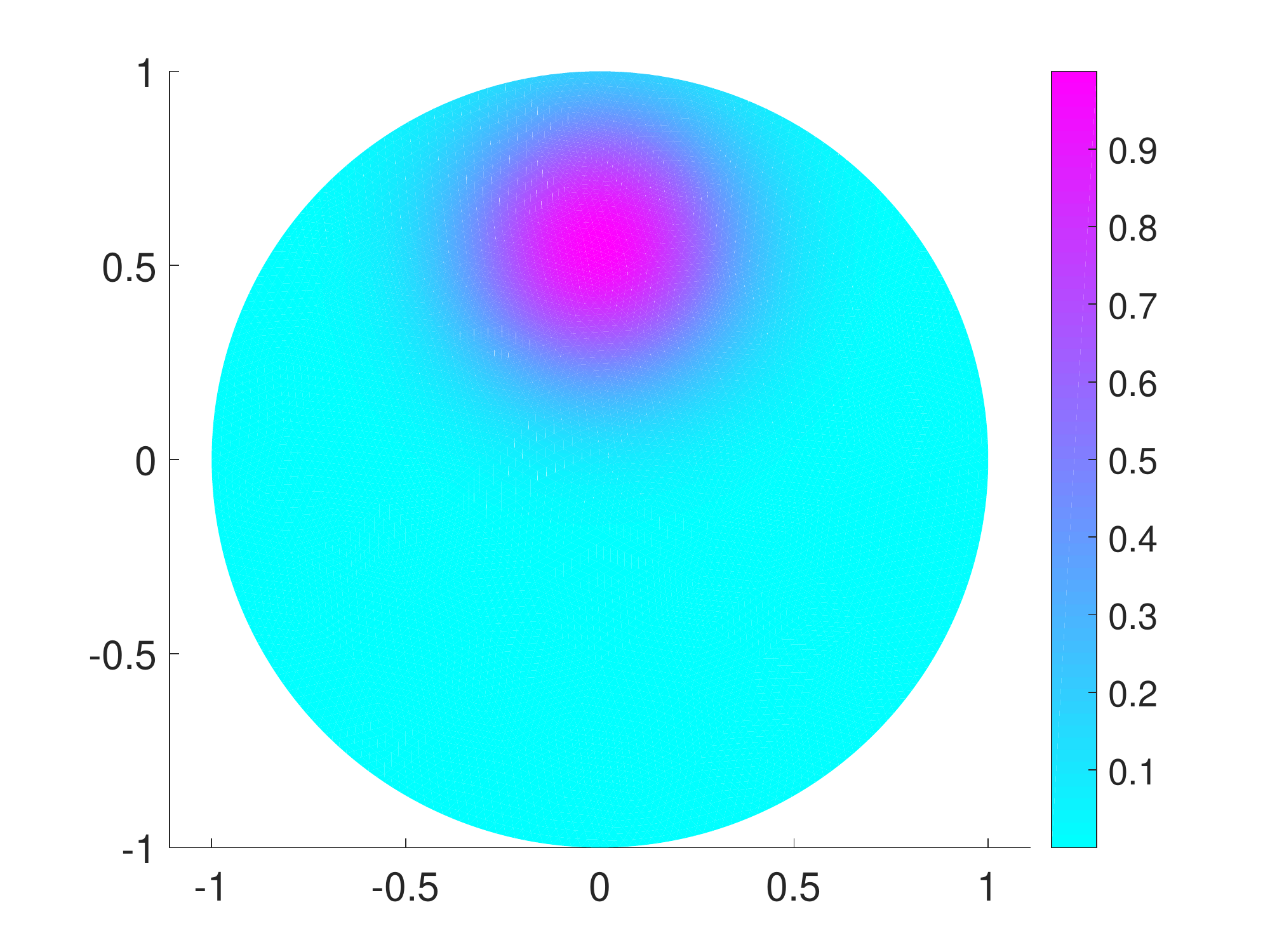} & \includegraphics[width=0.33\textwidth]{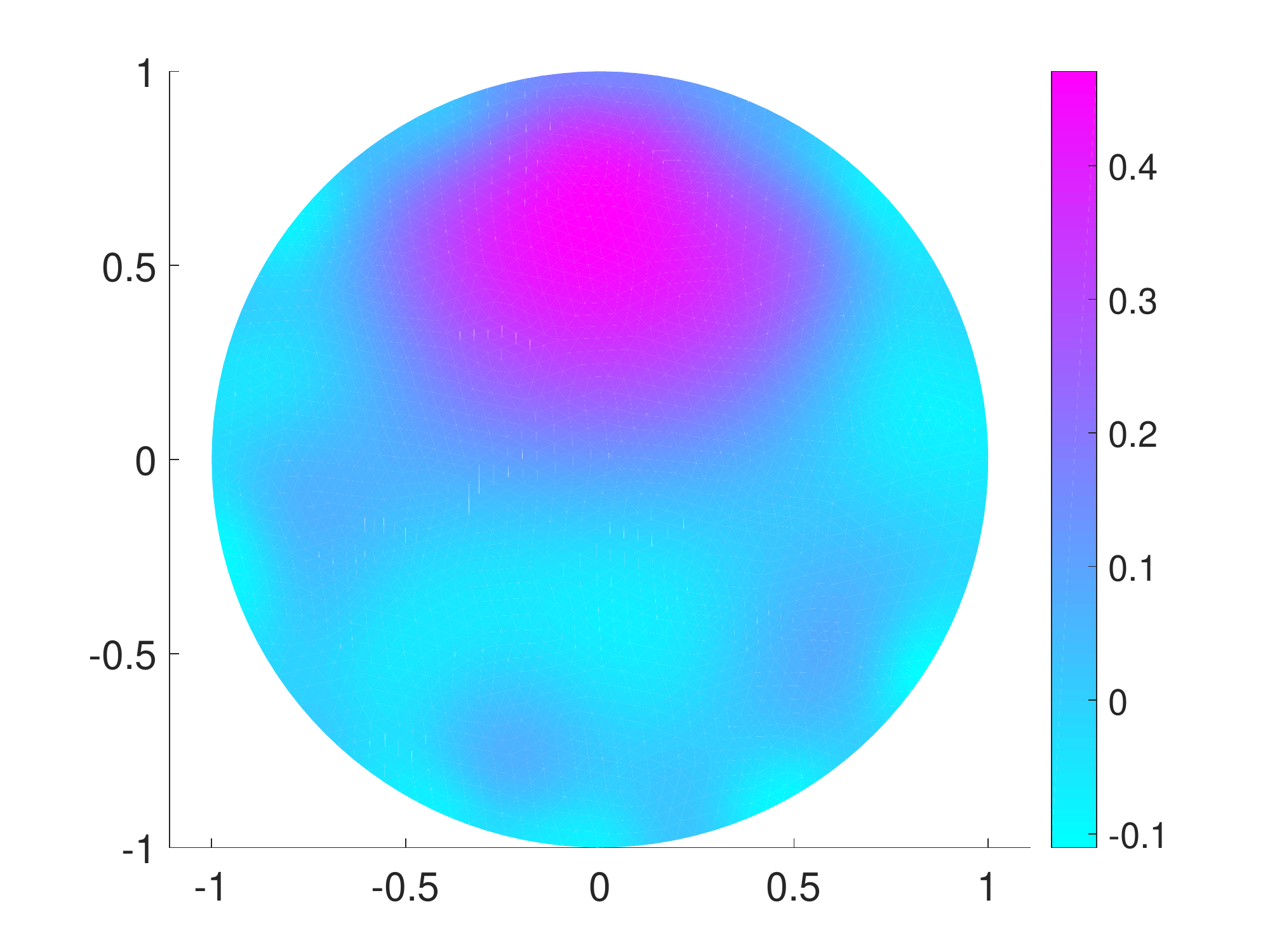} & \includegraphics[width=0.33\textwidth]{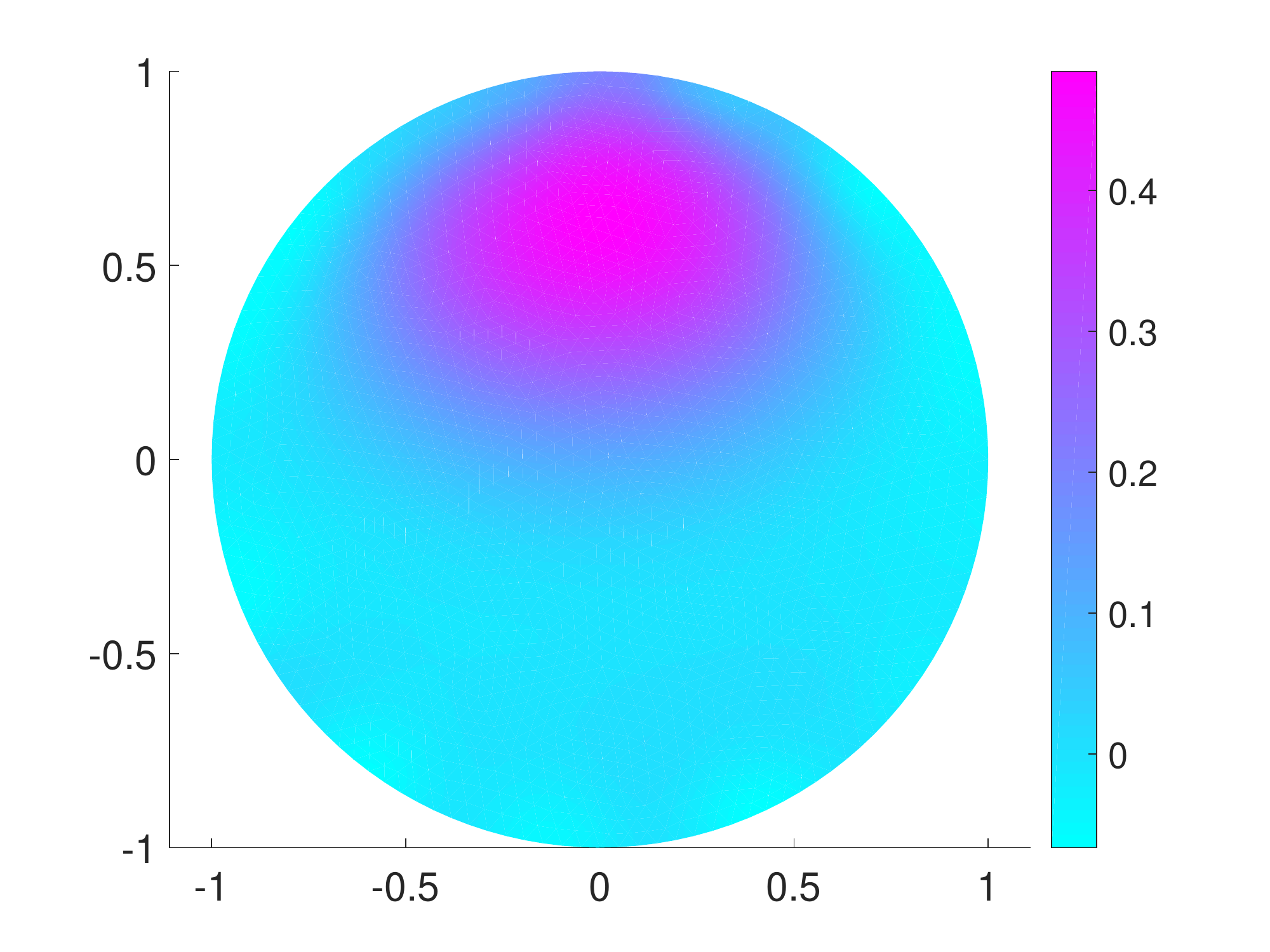}\\
   exact & linearization & RSVD
  \end{tabular}
  \caption{Numerical reconstructions for EIT with $0.1\%$ noise.\label{fig:eit}}
\end{figure}

\section{Conclusion}

In this work, we have provided a unified framework for developing efficient linear inversion techniques
via RSVD and classical regularization methods, building on a certain
range condition on the regularized solution. The construction is illustrated on three popular linear
inversion methods for finding smooth solutions, i.e., truncated singular value decomposition, Tikhonov regularization and general
Tikhonov regularization with a smoothness penalty. We have provided a novel interpretation of the approach
via convex duality, i.e., it first approximates the dual variable via randomized SVD and then recovers
the primal variable via duality relation. Further, we gave rigorous error bounds on the approximation
under the canonical sourcewise representation, which provide useful guidelines for constructing a low-rank
approximation. We have presented extensive numerical experiments, including nonlinear tomography, to illustrate the efficiency and accuracy
of the approach, and demonstrated its competitiveness with existing methods.


\begin{algorithm}[hbt!]
  \centering
  \caption{Iterative refinement of RSVD-Tikhonov solution. \label{alg:iter-refine}}
  \begin{algorithmic}[1]
    \STATE Give $A$, $b$ and $J$, and initialize $(x^0,p^0)=(0,0)$.
    \STATE Compute RSVD $(\tilde U_k,\tilde \Sigma_k,\tilde V_k)$
    to $AL^\dag$ by Algorithm \ref{alg:rsvd}.
    \FOR {$j=1,\ldots,J$}
      \STATE Compute the auxiliary variable $z^{j}$ by \eqref{eqn:it-z}.
      \STATE Update the dual variable $p^{j+1}$ by \eqref{eqn:it-p}.
      \STATE Update the primal variable $x^{j+1}$ by \eqref{eqn:it-x}.
      \STATE Check the stopping criterion.
    \ENDFOR
    \STATE Output $x^J$ as an approximation to $x_\alpha$.
  \end{algorithmic}
\end{algorithm}

\appendix

\section*{Appendix A: Iterative refinement}\label{ssec:duality}

Proposition \ref{prop:dual} enables iteratively refining the inverse solution when RSVD is not sufficiently accurate.
This idea was proposed in \cite{WangLeeMahdaviKolar:2017,ZhangMahdaviJinYang:2014} for standard Tikhonov regularization,
and we describe the procedure in a slightly more general context. Suppose $\mathcal{N}(L)=\{0\}$. Given a current
iterate $x^j$, we define a functional $J_\alpha^j(\delta x)$ for the increment $\delta x$ by
\begin{equation*}
  J_\alpha^j(\delta x) : = \|A(\delta x+x^j)-b\|^2 + \alpha \|L(\delta x+x^j)\|^2.
\end{equation*}
Thus the optimal correction $\delta x_\alpha$ satisfies
\begin{equation*}
  (A^*A + \alpha L^*L)\delta x_\alpha = A^*(b-Ax^j)-\alpha L^*Lx^j,
\end{equation*}
i.e.,
\begin{equation}\label{eqn:opt-cor}
  (B^* B + \alpha I) L\delta x_\alpha = B^* (b-Ax^j)-\alpha Lx^j,
\end{equation}
with $B=AL^\dag$.
However, its direct solution is expensive. We employ RSVD for a low-dimensional
space $\tilde V_k$ (corresponding to $B$), parameterize the increment $L\delta x$ by $L\delta x=\tilde
V_k^*z$ and update $z$ only. That is, we minimize the following functional
in $z$
\begin{equation*}
  J_\alpha^j(z) : = \|A(L^{\dag}\tilde V_k^*z+x^j)-b\|^2 + \alpha \|z+\tilde V_k L x^j\|^2.
\end{equation*}
Since $k\ll m$, the problem can be solved efficiently. More precisely, given the current estimate $x^j$, the optimal $z$ solves
\begin{equation}\label{eqn:it-z}
   (\tilde V_k B^*  B\tilde V_k^* + \alpha I) z = \tilde V_k B^*(b-A x^j)-\alpha \tilde V_kLx^j.
\end{equation}
It is the Galerkin projection of \eqref{eqn:opt-cor} for $\delta x_\alpha$
onto the subspace $\tilde V_k$. Then we update the dual $\xi$ and the primal
$x$ by the duality relation in Section \ref{ssec:duality}:
\begin{align}
  \xi^{j+1} & = b-Ax^j - B\tilde V_k^* z^j,\label{eqn:it-p}\\
  x^{j+1} & = \alpha^{-1}\Gamma A^*\xi^{j+1}.\label{eqn:it-x}
\end{align}
Summarizing the steps gives Algorithm \ref{alg:iter-refine}. Note that the duality relation
\eqref{eqn:Fenchel} enables $A$ and $A^*$ to enter into the play, thereby allowing progressively improving the accuracy.
The main extra cost lies in matrix-vector products by $A$ and $A^*$.

The iterative refinement is a linear fixed-point iteration, with
the solution $x_\alpha$ being a fixed point and the iteration matrix being independent of the
iterate. Hence, if the first iteration is contractive, i.e., $\|x^1-x_\alpha\| \leq c\|x^0-x_\alpha\|$,
for some $c\in(0,1)$, then Algorithm \ref{alg:iter-refine} converges linearly to
$x_\alpha$. It can be satisfied if the RSVD
approximation $(\tilde U_k,\tilde \Sigma_k,\tilde V_k)$ is reasonably accurate to $B$.

\bibliographystyle{siam}
\bibliography{bib_rsvd}
\end{document}